\documentclass[final,reqno]{siamltex}
\usepackage{latexsym,amsmath,amssymb,amsfonts,mathrsfs}
\usepackage{epsf,graphicx,epsfig,color,cite,cases,bm}
\usepackage{subfigure,graphics,multirow,marginnote,graphicx}
\usepackage{enumerate,dsfont}
\usepackage{algorithm}

\newcommand{\R}{{\mat R}}

\newcommand{\N}{{\mat N}}

\newcommand{\ds}{\displaystyle}
\newcommand{\no}{\nonumber}
\newcommand{\be}{\begin{eqnarray}}
\newcommand{\ben}{\begin{eqnarray*}}
\newcommand{\en}{\end{eqnarray}}
\newcommand{\enn}{\end{eqnarray*}}

\newcommand{\pa}{\partial}

\newcommand{\ov}{\overline}

\newcommand{\g}{\gamma}
\newcommand{\G}{\Gamma}

\newcommand{\vep}{\varepsilon}
\newcommand{\Om}{\Omega}

\newcommand{\mat}{\mathbb}
\newcommand{\se}{\setminus}

\begin{document}
\renewcommand{\theequation}{\arabic{section}.\arabic{equation}}


\title{\bf Near-field imaging of a locally rough interface and buried obstacles with the linear sampling method
}

\author{Jianliang Li\thanks{School of Mathematics and Statistics, Changsha University of Science
and Technology, Changsha, 410114, China ({\tt lijl@amss.ac.cn})}
\and
Jiaqing Yang\thanks{School of Mathematics and Statistics, Xi'an Jiaotong University,
Xi'an, Shaanxi 710049, China ({\tt jiaq.yang@mail.xjtu.edu.cn})}
\and
Bo Zhang\thanks{LSEC, NCMIS and Academy of Mathematics and Systems Science, Chinese Academy
of Sciences, Beijing 100190, China and School of Mathematical Sciences, University of Chinese
Academy of Sciences, Beijing 100049, China ({\tt b.zhang@amt.ac.cn})}
}
\date{}

\maketitle


\begin{abstract}
Consider the problem of inverse scattering of time-harmonic point sources from an infinite, penetrable
rough interface with bounded obstacles buried in the lower half-space, where the interface is assumed
to be a local perturbation of a planar surface. A novel version of the sampling method is proposed to
simultaneously reconstruct the local perturbation of the rough interface and buried obstacles by
constructing a modified near-field equation associated with a special rough surface, yielding a fast
imaging algorithm. Numerical examples are presented to illustrate the effectiveness of the inversion
algorithm.

\begin{keywords}
Inverse acoustic scattering, rough interfaces, buried obstacles, the linear sample method
\end{keywords}

\begin{AMS}
35R30, 35Q60, 65R20, 65N21, 78A46
\end{AMS}
\end{abstract}

\pagestyle{myheadings}
\thispagestyle{plain}
\markboth{Jiliang Li, Jiaqing Yang and Bo Zhang}
{Imaging of locally rough interface and buried obstacles}

\section{Introduction}
\setcounter{equation}{0}

This paper considers the problem of scattering of time-harmonic point sources by an infinite, penetrable
interface with buried obstacles below the interface. The direct and inverse scattering problems are studied,
where the direct scattering problem is to determine the distribution of the scattered wave in the whole
space when the incident wave, the interface and the buried obstacles with their boundary conditions
are given; while the inverse scattering problem aims to recover the shape and location of the unknown
interface and buried obstacles from the scattered wave measured on a bounded surface above the interface.
These problems have played a fundamental role in diverse scientific areas such as underwater exploration,
geophysical exploration and radar detection.

Precisely, we will consider the case where the rough interface is different from a plane over a finite
interval (called a locally rough interface), which separates the whole space into the upper and lower half-spaces.
The wave motion is then governed by the two-dimensional Helmholtz equation where the wavenumber is described
by a piecewise constant function. Across the interface, the total-field and its normal derivative are
assumed to be continuous, which also corresponds to the transverse electric polarization case.
For simplicity, we assume throughout this paper that the buried obstacles are sound-soft, which means that
the total-field vanishes on the boundary of the obstacles. Since the interface is only a local perturbation
of a plane surface, the Sommerfeld radiation condition remains valid to describe the behavior of the
scattered field away from the interface.
Scattering problems by unbounded rough surfaces have been studied by many authors for the case with no buried
obstacles via either the integral equation method or a variational approach
(see, e.g., \cite{SP05,WW87,BHY2018,SE10,GJ11,SRZ99,SZ99,QZZ19,ZZ13} and the references therein).
For the scattering problem considered in this paper its well-posedness has been established in our recent
work \cite{YLZ20} by reformulating the direct problem as the scattering problem by obstacles in a two-layered
medium together with the integral equation method.

In this paper, we are mainly interested in the inverse scattering problem which aims to simultaneously
recover the shape and location of the unknown interface and buried obstacles from the knowledge of the
scattered field measured on a line segment above the interface.
A global uniqueness result has been established in \cite{YLZ20} for this inverse scattering problem, 
that is, the locally rough interface, the wavenumber in the lower half-space and the buried obstacle 
along with its boundary condition can be uniquely determined by the scattered field measured on a line segment 
in the upper half-space and generated by point sources.
Based on this uniqueness result, we aim to develop an efficient algorithm to solve the inverse problem
numerically. If the rough interface is known, a numerical method has been given in \cite{AAY07} to detect 
the location of the buried obstacles from the far-field pattern, based on the determination of
the surface impedance that contains jumps at the surface points just above the buried obstacles.
If the rough interface is a plane, an inversion scheme was proposed to recover
the buried, impenetrable obstacles with phased far-field data in \cite{LLLL15}.
However, if the rough interface is an unknown perturbation 
of a plane, as far as we know, no inversion algorithm is available to reconstruct the rough interface
and the buried obstacle simultaneously.

So far, many inversion algorithms have been developed for solving inverse scattering problems by
unbounded rough surfaces for the case with no buried obstacles, such as 
the Kirsch-Kress method in \cite{CR10},
Newton-type algorithms in \cite{GJ11,CG11,SP02,QZZ19,ZZ2017,RM15, KT00, ZZ13} and the transformed field expansion based reconstruction algorithm 
in \cite{GL13,GL14} for the case when the surface is a small and smooth perturbation of a plane.
Note that the reconstruction methods in \cite{GL13, GL14, GJ11, QZZ19, ZZ13}
are iterative type methods under certain a priori knowledge on the rough surface.
On the other hand, several non-iterative type methods have also been proposed for the special case when
no buried obstacles exist, such as the singular source method for recovering
a perfectly conducting surface in \cite{C03,LC05}, the factorization method in \cite{AL08} for a Dirichlet
rough surface in the case $\kappa f_+\in (0,\sqrt{2})$, where $\kappa$ is the wave number and $f_+$
is the amplitude of the rough surface, the sampling type method for reconstructing an infinite,
locally rough interface in \cite{LYZ18}, extending the work in \cite{DLLY17} from the Dirichlet
impenetrable surface to the penetrable interface, and the direct imaging method for both impenetrable 
and penetrable surfaces in \cite{LZZ18,Zhang20}, 
and for a perfectly conducting locally rough surface with phaseless data in \cite{XZZ19}.  

In this paper, we will develop a linear sampling method (LSM) to solve the inverse scattering problem 
considered numerically, that is, to reconstruct both the local perturbation of the rough interface and 
the buried obstacles from the scattered near-field measurements in the upper half-space. 
Recently in \cite{LYZ18}, we developed a linear sampling method to reconstruct the local perturbation 
of the rough interface for the case with no buried obstacles, based on reformulating the original 
scattering problem by a local rough interface into an equivalent integral equation formulation in a 
bounded domain with the help of the free-space Green function associated with a class of special 
rough surfaces. However, the idea in \cite{LYZ18} does not work to reconstruct the interface and the 
buried obstacles simultaneously. To overcome this difficulty, we consider the difference between 
the solution to the original scattering problem and the solution to the scattering problem by a plane 
interface. Then the original scattering problem can be reduced to the scattering problem by an 
inhomogeneous medium of compact support and the buried obstacles. Based on this, we can construct 
a near-field equation and present a novel version of LSM to numerically recover the rough interface and 
the buried obstacles simultaneously. As far as we know, this is the first sampling type method 
for reconstructing both the rough interface and the buried obstacles simultaneously.  

The remaining part of this paper is organized as follows. Section \ref{sec2} presents some basic notations 
and function spaces used in the paper. Section \ref{sec3} introduces the mathematical formulation 
of the scattering problem of an incident point source by a locally rough interface with a buried obstacle 
below the interface. In Section \ref{sec3}, we also reformulate the original scattering problem as the one 
by an inhomogeneous medium of compact support and the buried obstacle, based on the difference between the 
solution to the original scattering problem and the solution to the scattering problem by a plane interface.
A novel version of the classical LSM is proposed in Section \ref{sec4} for solving the inverse problem via  
constructing a modified near-field equation. Numerical results are presented in Section \ref{sec5} to illustrate  
the validity of the inversion algorithm. Conclusions are given in Section \ref{sec6}.

\section{Preliminaries}\label{sec2}

In this section, we introduce some basic notations and function spaces used in this paper.
\begin{figure}[htbp]
\centering
\includegraphics[width=4in, height=2in]{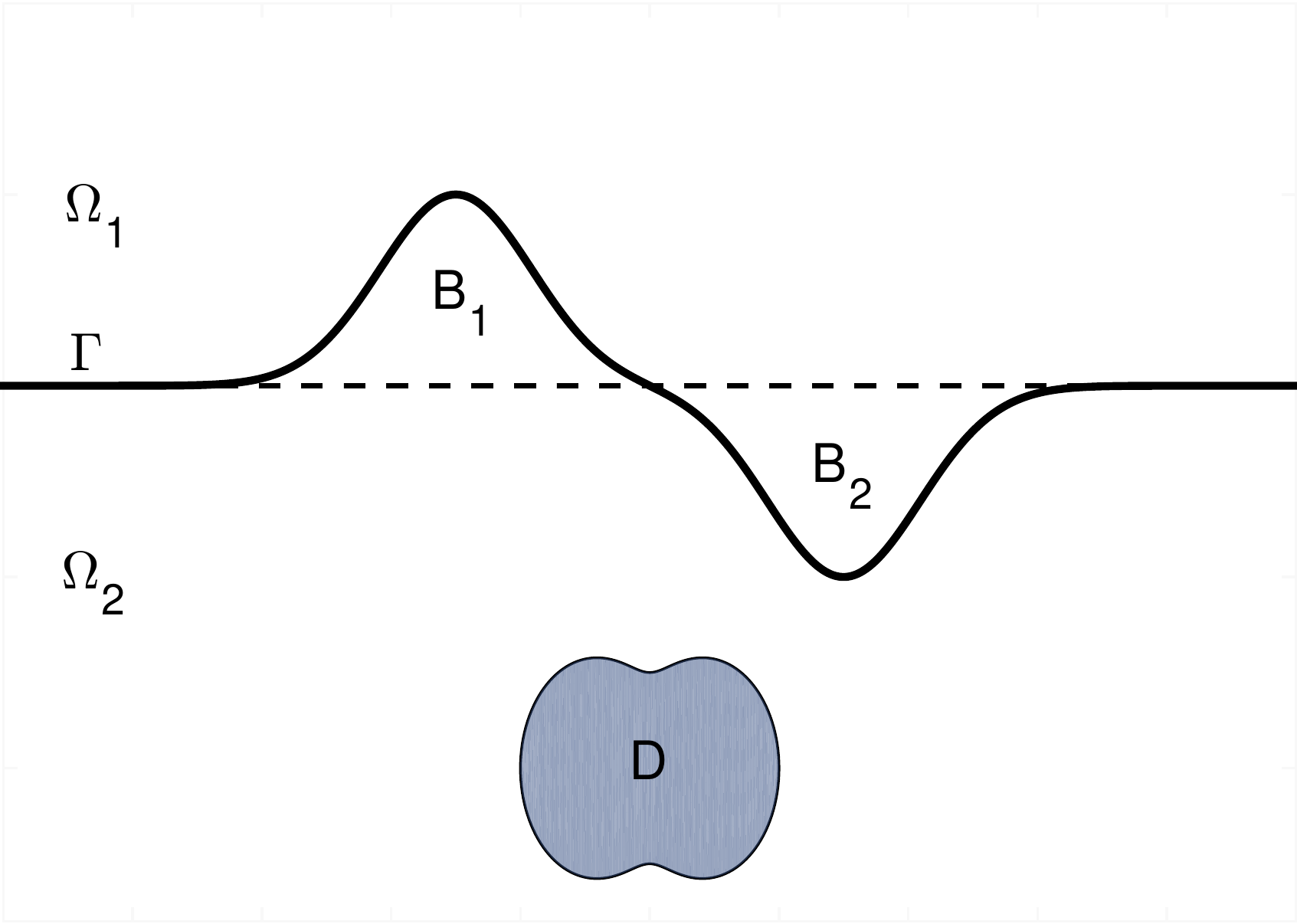}
\caption{The physical configuration of the scattering problem.}
\label{f1}
\end{figure}
As seen in Figure \ref{f1}, the scattering interface is described by a smooth curve
$\G:=\{(x_1, x_2)\in\R^2: x_2=f(x_1)\}$. Here, $f$ is assumed to be a Lipschitz continuous function which
is different from $0$ over a finite interval. It means that $\G$ is just a local perturbation of the planar
interface $\G_0:=\{(x_1,x_2)\in\R^2:x_2=0\}$. Then the whole space is
separated by $\G$ into two unbounded half-spaces $\Omega_1:=\{(x_1, x_2)\in\R^2: x_2>f(x_1)\}$ and
$\Omega_2:=\{(x_1, x_2)\in\R^2: x_2<f(x_1)\}$. Denoted by $B_1:=\R^2_{+}\cap\Omega_2$ the bounded domain
above $\G_0$ and $B_2:=\R^2_-\cap\Omega_1$ the bounded domain below $\G_0$, where $\R^2_\pm:=\{x_2\gtrless0\}$.
For simplicity, we consider in this paper a simple case that $\G$ has only two local perturbations;
see Figure \ref{f1}. The results obtained can be easily extended to the case of multiple local perturbations.
Moreover, the embedded obstacle $D$ is considered to be buried into the lower half-space $\Om_2$
with a smooth boundary $\pa D \in C^{2,\alpha}$ for some H\"{o}lder exponent $0<\alpha\leq 1$.

For a bounded domain $\Om\subset\R^2$ with Lipschitz boundary $\pa\Om$, denote by $W^{m,p}(\Om)$
the usual Sobolev space for index $m\in\N$ and $p\in [1,\infty)$, which is a Banach space
with respect to the $W^{m,p}$-norm
\ben
||u||_{m,p}:=\left(\sum\limits_{|\alpha|\leq m} ||\pa^{\alpha}u||^p_{p}\right)^{1/p}.
\enn
If $p=2$, it is also conventional to write $H^m:=W^{m,2}$, which is a Hilbert space under the inner product
$(u,v)_m: = \sum_{|\alpha|\leq m}(\pa^\alpha u,\pa^\alpha v)_{L^2}$. It is easily observed
that the space $W^{0,2}(\Om)$ coincides with the usual space $L^2(\Om)$ consisting of all square-integrable
functions on $\Om$.

Moreover, let $L^2_{\Delta}(\Omega):=\{u\in\mathcal{D'}(\Om)|u\in L^2(\Om),\;\Delta u\in L^2(\Om)\}$.
Clearly, it is also a Hilbert space equipped under the inner-product
\ben
(u,v)_{L^2_{\Delta}(\Om)}=(u,v)_{L^2(\Om)}+(\Delta u,\Delta v)_{L^2(\Om)}
\quad{\rm for\;}u,v\in L^2_{\Delta}(\Om).
\enn
Let $\g_ju:=\frac{\pa^j u}{\pa\nu^j}\big|_{\pa\Om}$ $(j=0,1)$ denote the trace maps defined
on $C^{\infty}(\ov{\Om})$, where $\nu$ is the outward unit normal vector to $\pa\Om$.
By the Sobolev embedding theorem, $\g_j$ has a continuous extension from $L^2_{\Delta}(\Om)$
into $H^{-j-1/2}(\pa\Om)$, that is, there exists a fixed constant $C>0$ such that
\ben
||\g_ju||_{H^{-j-1/2}(\pa\Omega)}\leq C||u||_{L^2_{\Delta}(\Omega)}
\enn
for $j=0,1$.

Furthermore, we introduce the subspaces $X_j$ of $L^2_{\Delta}(B_j)$ by
\be\label{a1}
X_j:=\{u\in L^2_{\Delta}(B_j),\;\Delta u+\kappa_j^2 u=0\quad{\rm in}\;B_j\}
\en
for $j=1,2$. In (\ref{a1}) the equation $\Delta u+\kappa_j^2 u=0$ is understood in the distribution
with the two positive constants $\kappa_1>0$ and $\kappa_2>0$.  It is easily verified that $X_j$ is
a closed subspace of $L^2_{\Delta}(B_j)$ and is thus a Hilbert space. Let $\chi_{j}$ denote the
characterization function of the domain $B_j$, defined by $\chi_j=1$ in $B_j$ and vanishes outside $B_j$.
Then, we can define the space
\ben\label{a2}
X:=\{u=\chi_1u_1+\chi_2u_2: \;u_1\in X_1, u_2\in X_2\},
\enn
which is also a Hilbert space equipped with the inner product $(u,v)_X=(u_1,v_1)_{L^2_{\Delta}(B_1)}+(u_2,v_2)_{L^2_{\Delta}(B_2)}$ for $u, v\in X$.

\section{Mathematical formulation}\label{sec3}
\setcounter{equation}{0}

In this section, we introduce the mathematical description on the scattering of time-harmonic point sources
from a locally rough interface with a buried obstacle in the lower half-space.
Suppose the incident field $u^{\rm i}(\cdot,y)$ induced by a point source
\ben
\Phi_{\kappa_1}(x,y) = \frac{\rm i}{4}H_0^{(1)}(\kappa_1|x-y|),\quad x\not=y,\;y\in\Om_1,
\enn
which corresponds to a fundamental solution to the Helmholtz equation $\Delta \Phi_{\kappa_1}(\cdot,y)+\kappa_1^2\Phi_{\kappa_1}(\cdot,y)=-\delta_y$ in $\R^2$. Here, $H_0^{(1)}(\cdot)$
is the Hankel function of the first kind of order zero. Then the scattering of $u^{\rm i}(\cdot,y)$ by $\G$
and $D$ can be modeled by the two-dimensional Helmholtz equation
 \be\label{a4}
 \Delta u+\kappa^2u=-\delta_y\quad{\rm in}\;\;\R^2\se\ov{D},
 \en
where  $\kappa>0$ is the wavenumber defined as $\kappa:=\kappa_1$ in $\Omega_1$ and $\kappa:=\kappa_2$
in $\Omega_2$, and $u$ is the total field defined as $u:=u^{\rm i}+u^{\rm s}$
in $\Om_1$ and $u:=u^{\rm s}$ in $\Om_2\se\ov{D}$.

For simplicity, we consider that $u$ satisfies a Dirichlet boundary condition on $\pa D$:
\be\label{a7}
u=0\qquad{\rm on}\; \pa D,
\en
which means that $D$ is a sound-soft obstacle.
Since $\G$ is just a local perturbation of $\G_0$, the scattered field $u^{\rm s}$ satisfies the so-called
Sommerfeld radiation condition
\be\label{a8}
\lim_{r\rightarrow\infty}\sqrt{r}\left(\frac{\pa u^{\rm s}}{\pa r}-{\rm i}\kappa u^{\rm s}\right)=0,\quad r=|x|.
\en
It should be pointed out that the condition (\ref{a8}) can be replaced by the much weaker Upward and
Downward Propagating Radiation Conditions (UPRC and DPRC). We refer the reader to \cite{SRZ99,SZ99,DLLY17}
for detailed discussions on the UPRC and DPRC for nonlocal surfaces.

The well-posedness of (\ref{a4})-(\ref{a8}) has been extensively studied in the literature for the case
$D=\emptyset$; see e.g., \cite{SE10,LYZ13, DTS03,MT06,ZZ13}. Moreover, it is noteworthy that a different method
from the previous works was recently proposed by the current authors in \cite{YLZ20}, where
the well-posedness of Problem (\ref{a4})-(\ref{a8}) was shown in Sobolev spaces based on an equivalent
Lippmann-Schwinger type integral equation defined in a bounded domain.

In this paper, we are concerned with the inverse problem of simultaneously recovering the rough interface $\G$ and the buried obstacle $D$ from near-field measurements above $\G$. We refer to \cite{YLZ20}
with a global uniqueness theorem for determining all unknown $\G$, $D$ and $\kappa_2$. Based on this, we aim to study its numerical solution by proposing a valid sampling-type method.
Since $\G$ is a local perturbation of $\G_0$, the recovery of $\G$ can be reduced to the recovery of the local perturbations $B_j$. It means that our inverse problem can be reduced to distinguish two different local perturbations $B_1$ and $B_2$ and the buried obstacle $D$. To this end, we first introduce the fundamental solution $G_0(x, y)$ of the
two-dimensional Helmholtz equation in a two-layered background medium separated by $\G_0$, which satisfies
\be\label{a9}
\left\{\begin{array}{ll}
\Delta_x G_0(x,y)+\kappa^2_0G_0(x,y)=-\delta_y &\qquad{\rm in}\; \R^2\\ [3mm]
\ds\lim_{r\rightarrow\infty}\sqrt{r}\left(\frac{\pa G_0(x,y)}{\pa r}-{\rm i}\kappa_0G_0(x,y)\right)=0&
\qquad {\rm for}\; r=|x|
\end{array}\right.
\en
with the wavenumber $\kappa_0:=\kappa_1$ in $\R^2_+$ and $\kappa_0:=\kappa_2$ in $\R^2_-$.
It follows from \cite{L10} that $G_0(x,y)$ has an explicit expression
\ben
&&G_0(x,y):=\Psi^{(1)}(x,y)+\Phi_1(x,y)\qquad{\rm for}\; x\in \R^2_+,\;\;y\in\R^2_+,\\
&&G_0(x,y):=\Psi^{(2)}(x,y)\qquad\;\;\;\qquad\qquad{\rm for}\; x\in \R^2_-,\;\;y\in\R^2_+,
\enn
where
\ben
&&\Psi^{(1)}(x,y)=\frac{\rm i}{4\pi}\int_{\R}\frac{1}{\beta_1}\frac{\beta_1-\beta_2}{\beta_1+\beta_2}e^{{\rm i}\beta_1(x_2+y_2)}e^{{\rm i}\xi (x_1-y_1)}{\rm d}\xi\\
&&\Psi^{(2)}(x,y)=\frac{\rm i}{2\pi}\int_{\R}\frac{e^{{\rm i}(\beta_1y_2-\beta_2x_2)}}{\beta_1+\beta_2}e^{{\rm i}\xi (x_1-y_1)}{\rm d}\xi\\
\enn
with $\beta_j$ given by
\ben
\beta_j=\left\{\begin{array}{l}
                 \sqrt{\kappa_j^2-\xi^2}\qquad\; {\rm for}\; |\kappa_j|>|\xi|,\\
                 {\rm i}\sqrt{\xi^2-\kappa_j^2}\qquad {\rm for}\; |\kappa_j|<|\xi|.
               \end{array}
\right.
\enn
Applying the dominated convergence theorem to $\Psi^{(1)}(x,y)$, $\Psi^{(2)}(x,y)$ and
their derivatives, we have $\Psi^{(1)}(x,y),\Psi^{(2)}(x,y)\in C^{\infty}(\R^2\times\R^2\se\{y\}\times\{y\})$.

For a fixed $y\in\R^2_+$, it is easily found that $G_0(x,y)$ contains the information of $\G_0$,
and the solution $u(x,y)$ to (\ref{a4})-(\ref{a8}) contains the information of the scattering surface $\G$
and the buried obstacle $D$. Thus, we consider the difference $\tilde{u}(x,y):=u(x,y)-G_0(x,y)$
which clearly contains all unknown information of $B_1$, $B_2$ and $D$.
It then follows from the Helmholtz equations for $G_0$ and $u$ that $\tilde{u}$ satisfies
\be\label{a11}
\Delta\tilde{u}+\kappa^2\tilde{u}= g\quad{\rm in}\;\;\R^2\se\ov{D},
\en
and the Sommerfeld radiation condition (\ref{a8}), where $g: = \eta(\chi_1-\chi_2)G_0$
with $\eta:=\kappa_1^2-\kappa_2^2$. The Dirichlet boundary condition (\ref{a7}) for $u$ gives
\be\label{a12}
\tilde{u} = -G_0\qquad{\rm on}\; \pa D.
\en
With the above analysis, Problem (\ref{a11})-(\ref{a12}) can be regarded as a special case of
the following boundary value problem of finding $\tilde{u}\in H^1_{\rm loc}(\R^2\se\ov{D})$ such that
\be\label{a13}
\left\{\begin{array}{lll}
\Delta \tilde{u}+\kappa^2\tilde{u}=h_1& \qquad {\rm in}\; \R^2\se\ov{D}\\ [1mm]
\tilde{u}=h_2&\qquad{\rm on}\; \pa D\\ [1mm]
\ds\lim_{r\rightarrow\infty}\sqrt{r}\left(\frac{\pa\tilde{u}}{\pa r}-{\rm i}\kappa\tilde{u}\right)
=0& \qquad {\rm for}\;\; r=|x|
\end{array}\right.
\en
for $h_1\in X$ and $h_2\in Y$. Here, $Y:=\{\g_0 u:\; u\in H^1(D)\;\mbox{satisfies}\;\Delta u+\kappa_2^2u
=0\;{\rm in}\;D\}.$

Next, we shall prove the existence of a unique solution to Problem (\ref{a13}) in $H^1_{\rm loc}(\R^2\se\ov{D})$. The uniqueness follows directly from the uniqueness of the scattering problem (\ref{a4})-(\ref{a8}). For the existence, we decompose $\tilde{u}$ into two parts: $\tilde{u}:=\tilde{u}_1+\tilde{u}_2$ with $\tilde{u}_1$ and $\tilde{u}_2$ satisfying
\be\label{a14}
\left\{\begin{array}{lll}
\Delta \tilde{u}_1+\kappa^2\tilde{u}_1=h_1& \qquad {\rm in}\; \R^2\se\ov{D}\\ [1mm]
\tilde{u}_1=0&\qquad{\rm on}\; \pa D\\ [1mm]
\ds\lim_{r\rightarrow\infty}\sqrt{r}\left(\frac{\pa\tilde{u}_1}{\pa r}-{\rm i}\kappa\tilde{u}_1\right)
=0&\qquad {\rm for}\;\;r=|x|
\end{array}\right.
\en
and
\be\label{a15}
\left\{\begin{array}{lll}
\Delta \tilde{u}_2+\kappa^2\tilde{u}_2=0& \qquad {\rm in}\; \R^2\se\ov{D}\\ [1mm]
\tilde{u}_2=h_2&\qquad{\rm on}\; \pa D\\ [1mm]
\ds\lim_{r\rightarrow \infty}\sqrt{r}\left(\frac{\pa\tilde{u}_2}{\pa r}-{\rm i}\kappa\tilde{u}_2\right)
=0& \qquad {\rm for}\;\; r=|x|
\end{array}\right.
\en
respectively.

Recall that $h_1\in X$, and so there exists $h_{1j}\in X_j$, $j=1,2$, such that $h_1=\chi_1 h_{11}+\chi_2 h_{12}$.
By the standard discussion, it can be deduced that the unique solution to Problem (\ref{a14} ) can be
written in the following form
\ben\label{a16}
\tilde{u}_1(x)=-\int_{B_1}\Psi(x,y)h_{11}(y){\rm d}y-\int_{B_2}\Psi(x,y)h_{12}(y){\rm d}y\qquad{\rm for}\; x\in\R^2\se\ov{D},
\enn
where $\Psi(x,y)$ denotes the associated Green's function to Problem (\ref{a4})-(\ref{a8}).
For the existence of $\Psi(x,y)$, we refer to, e.g., \cite{YLZ20} for a detailed discussion.

Next, we shall prove the existence of a unique solution to Problem (\ref{a15}) by employing the boundary
integral equation technique. To this end, let $G_{\G}(x,y)$ be the fundamental solution of the two-dimensional
Helmholtz equation in a two-layered background medium separated by $\G$ such that
\be\label{a17}
\left\{\begin{array}{ll}
\Delta_x G_{\G}(x,y)+\kappa^2G_{\G}(x,y)=-\delta_y &\qquad{\rm in}\; \R^2\\ [3mm]
\ds\lim_{r\rightarrow\infty}\sqrt{r}\left(\frac{\pa G_{\G}(x,y)}{\pa r}-{\rm i}\kappa G_{\G}(x,y)\right)=0
& \qquad {\rm for}\; r=|x|.
\end{array}\right.
\en
For $h_2\in Y$, we seek a solution in the form of a combined double- and single-layer potential
\ben\label{a19}
\tilde{u}_2(x)=\int_{\pa D}\left(\frac{\pa G_{\G}(x,y)}{\pa\nu(y)}-{\rm i}G_{\G}(x,y)\right)\psi(y){\rm d}s(y)
\quad{\rm for}\; x\in\R^2\se\ov{D}
\enn
with density $\psi\in H^{1/2}(\pa D)$.

In view of (\ref{a17}), $\tilde{u}_2$ automatically satisfies the Helmholtz equation and
the Sommerfeld radiation condition. Following the boundary condition on $\pa D$, it is seen
that $\tilde{u}_2$ is a solution of Problem (\ref{a15}) if $\psi$ is a solution of the following equation
\ben\label{a20}
(I+K-{\rm i}S)\psi = 2h_2,
\enn
where $S$ and $K$ are the single- and double-layer operators given by
\ben\label{a21}
&&(S\psi)(x):=2\int_{\pa D}G_{\G}(x,y)\psi(y){\rm d}s(y),\qquad\quad x\in\pa D,\\\label{a22}
&&(K\psi)(x):=2\int_{\pa D}\frac{\pa G_{\G}(x,y)}{\pa \nu(y)}\psi(y){\rm d}s(y),\qquad x\in\pa D.
\enn
It is known by corollary 3.7 in \cite{CK13} as well as $\pa D\in C^{2,\alpha}$ for
some H\"{o}lder exponent $0<\alpha\leq 1$ that both $S$ and $K$ are bounded from $H^{1/2}(\pa D)$
into $H^{1}(\pa D)$.
Using the compact embedding of $H^{1}(\pa D)$ into $H^{1/2}(\pa D)$, it is concluded
that $I+K-{\rm i}S: H^{1/2}(\pa D)\rightarrow  H^{1/2}(\pa D)$ is Fredholm type with index $0$.
By a similar argument as in the proof of Theorem 3.11 in \cite{CK13}, and with the aid of the classical
Riesz-Fredholm theory we can prove the existence of $\psi$.
These results are summarized in the following theorem.

\begin{theorem}\label{theorem1}
If $\pa D\in C^{2,\alpha}$ with some H\"{o}lder exponent $0<\alpha\leq 1$,
then, for $(h_1, h_2)\in X\times Y$ the problem (\ref{a13}) has a unique solution
$\tilde{u}\in H^1_{\rm loc}(\R^2\se\ov{D})$ satisfying that
\ben
\|\tilde{u}\|_{H^1((\R^2\se\ov{D})\cap B_R)}\leq C_R(\|h_1\|_X+\|h_2\|_Y)
\enn
for some positive constant $C_R>0$ depends on $R>0$, where $B_R:=\{x\in\R^2:|x| <R\}$.
\end{theorem}

By Theorem \ref{theorem1}, we define the solution operator $L: X\times Y\rightarrow L^2(\G_{b,a})$
of Problem (\ref{a13}) by
\ben\label{a23}
L(h_1, h_2)=\tilde{u}|_{\G_{b,a}},
\enn
where $\G_{b,a}:=\{x=(x_1,x_2)\in\R^2, |x_1|\leq a, x_2=b\}$ with $a,b>0$. It follows from
Theorem \ref{theorem1} that $L$ is bounded from $X\times Y$ into $L^2(\G_{b,a})$.

For further analysis for our sampling method,
we introduce the associated interior transmission problem (ITP) of finding a pair of
functions $(v,w)\in L^2(\Omega)\times L^2(\Omega)$ satisfying
\be\label{a24}
\Delta v+\kappa_1^2v=0,\qquad \Delta w+\kappa_1^2q(x) w=0\qquad{\rm in}\;\Om,
\en
and the transmission conditions
\be\label{a25}
w-v=G_0(\cdot,z),\qquad \frac{\pa w}{\pa \nu}-\frac{\pa v}{\pa \nu}
=\frac{\pa G_0(\cdot,z)}{\pa\nu}\qquad{\rm on }\;\;\pa\Om,
\en
where $z\in \Omega$ and $q(x):=\kappa_2^2/\kappa_1^2$.
Recall that the values $\kappa_1^2>0$ is called a transmission eigenvalue if the homogeneous ITP
has nonzero solutions $(v,w)\in L^2(\Om)\times L^2(\Om)$ with $v-w\in H_0^2(\Omega)$.
It was shown in \cite{CGH10} that there exists an infinite discrete set of transmission eigenvalues.
It was further shown that if $\kappa_1^2$ is not transmission eigenvalues then there exists a unique
solution $(v,w)\in L^2(\Om)\times L^2(\Om)$ to the problem (\ref{a24})-(\ref{a25})
with $w-v\in H^2(\Omega)$ for each $z\in\Om$.

\begin{lemma}\label{lemma1}
If $\kappa^2_1$ is neither transmission eigenvalue of $B_1$ nor that of $B_2$, then $L$ is injective.
\end{lemma}

\begin{proof}
Let $\tilde{u}\in H^1_{\rm loc}(\R^2\se\ov{D})$ be the solution to Problem (\ref{a13})
for $(h_1,h_2)\in X\times Y$. If $L(h_1,h_2)=0$, we have $\tilde{u}=0$ on $\G_{b,a}$.
The analyticity of $\tilde{u}$ implies that $\tilde{u}$ vanishes on $\G_b:=\{x\in\R^2:x_2=b\}$.
Then it is concluded that $\tilde{u}=0$ in $U_b:=\{x\in\R^2:x_2>b\}$ from the uniqueness of
the Dirichlet problem. The analytic continuation discussion shows that $\tilde{u}=0$ in $\R^2_+\cap\Omega_1$.
By noting the continuity of the Cauchy data $(\tilde{u},\frac{\pa\tilde{u}}{\pa\nu})$ across $\G_0\cap\G$,
it is obtained by the unique continuation principle that $\tilde{u}=0$
in $\R^2\se\{\ov{B_1}\cup\ov{B_2}\cup\ov{D}\}$. Hence, $h_2=0$.

Since $h_1\in X$, there exists $h_{11}\in X_1$ and $h_{12}\in X_2$ such that $h_1=\chi_1h_{11}+\chi_2h_{12}$.
Thus, it suffices to show that $h_{11}=0$ and $h_{12}=0$. To this end, define $\psi_1:=\eta^{-1}h_{11}$
and $\psi_2:=\tilde{u}+\psi_1$. It follows from (\ref{a1}), (\ref{a13}) and the continuity of
$(\tilde{u},{\pa\tilde{u}}/{\pa\nu})$ that $(\psi_1,\psi_2)$ solves the homogeneous form of
the interior transmission problem (\ref{a24})-(\ref{a25}) in $B_1$. This yields that $(\psi_1,\psi_2)=0$
since $\kappa_1^2$ is not a transmission eigenvalue in $B_1$. Therefore, $h_{11}=0$. Define $\varphi_1:=\tilde{u}+\varphi_2$ with $\varphi_2:=-\eta^{-1}h_{12}$. Then we have by a similar argument
that $h_{12}=0$. The proof is thus complete.
\end{proof}

\begin{lemma}\label{lemma2}
If $\kappa^2_1$ is neither transmission eigenvalue of $B_1$ nor that of $B_2$ and $\kappa_2^2$ is not a
Dirichlet eigenvalue for $-\Delta$ in $D$, then $G_0(x,z)|_{\G_{b,a}}\in{\rm Range}(L)$ if and only if
$z\in D\cup B_1\cup B_2$.
\end{lemma}

\begin{proof}
We first prove the assertion that $z\in D\cup B_1\cup B_2$ if $G_0(x, z)|_{\G_{b,a}}\in {\rm Range}(L)$.
Assume on the contrary that $z\notin D\cup B_1\cup B_2$. Since $G_0(x, z)|_{\G_{b,a}}\in {\rm Range}(L)$,
there exists some $(\hat{h}_1,\hat{h}_2)\in X\times Y$ such that $L(\hat{h}_1,\hat{h}_2)=G_0(x,z)|_{\G_{b,a}}$.
Let $\hat{u}$ denote the solution to Problem (\ref{a13}) with the data $(h_1,h_2):=(\hat{h}_1,\hat{h}_2)$.
Then $L(\hat{h}_1,\hat{h}_2)=\hat{u}|_{\G_{b,a}}=G_0(x, z)|_{\G_{b,a}}$. A similar argument as in the proof of
Lemma \ref{lemma1} gives that $\hat{u}(x)=G_0(x,z)$ for $x\in\R^2\se\{\ov{D}\cup\ov{B_1}\cup\ov{B_{2}}\cup\{z\}\}$.
It is impossible since $G_0(x,z)$ is singular at $x=z$ and $\hat{u}(x)$ is smooth at $x=z$.
Hence, $z\in D\cup B_1\cup B_2.$

We next prove that $G_0(x,z)|_{\G_{b,a}}\in{\rm Range}(L)$ if $z\in D\cup B_1\cup B_2$. To this end,
we construct a function $\tilde{u}\in H^1_{\rm loc}(\R^2\se\ov{D})$ satisfying Problem (\ref{a13})
with $\tilde{u}(x):=G_0(x,z)$ on $\G_{b,a}$. Since $z\in D\cup B_1\cup B_2$, we consider the following ITP
\be\label{a26}
\left\{\begin{array}{llll}
\Delta v_j+\kappa_1^2v_j=0& \qquad {\rm in}\; B_j,\\ [1mm]
\Delta w_j+\kappa_2^2w_j=0& \qquad {\rm in}\; B_j,\\ [1mm]
w_j-v_j=q_{1j} &\qquad {\rm on}\;\pa B_j,\\ [1mm]
\frac{\pa w_j}{\pa\nu}-\frac{\pa v_j}{\pa \nu}=q_{2j}&\qquad{\rm on}\; \pa B_j
\end{array}\right.
\en
for $j=1,2$, with the boundary data $q_{1j}:=G_0(\cdot,z)|_{\pa B_j}$ and
$q_{2j}:=\pa_\nu G_0(\cdot,z)|_{\pa B_j}$, where $\nu$ denotes the outward unit normal vector to $\pa B_j$.
Since $\kappa^2_1$ is not an eigenvalue, there exists a unique solution $(v_j,w_j)\in L^2(B_j)\times L^2(B_j)$
to the ITP (\ref{a26}) with $w_j-v_j\in H^2(B_j)$. Now we construct $\tilde{u}$ as follows
\be\label{a27}
\tilde{u}:=\left\{\begin{array}{l}
              G_0(\cdot,z)\qquad\; \textrm{in}\; \R^2_+\cap\Omega_1, \\ [1mm]
              w_j-v_j \quad\;\;\;\;\; \textrm{in}\; B_j,\\ [1mm]
              G_0(\cdot,z) \qquad\; \textrm{in}\; \R^2_-\cap\Omega_2\se\ov{D}.
            \end{array}
\right.
\en
A direct calculation yields that $\tilde{u}\in H^1_{\rm loc}(\R^2\se\ov{D})$ satisfies
Problem (\ref{a13}) with the data
\ben\label{a28}
h_1=\eta(\chi_1v_1+\chi_2w_2)\quad {\rm and}\quad h_2=G_0(\cdot,z)|_{\pa D}.
\enn
We claim that $(h_1,h_2)\in X\times Y$. First it is seen by the Helmholtz equations for $v_1$ and $w_2$
that $h_1\in X$ if $z\in D\cup B_1\cup B_2$. To show that $h_2\in Y$, we have to distinguish
two cases that $z\in B_1\cup B_2$ and $z\in D$, respectively.

If $z\in B_1\cup B_2$, it is clear that $G_0(\cdot,z)\in H^1(D)$ solves
$\Delta G_0(\cdot,z)+\kappa_2^2G_0(\cdot,z)=0$ in $D$, which implies that $ h_2=G_0(\cdot,z)|_{\pa D}\in Y$.
If $z\in D$, we consider the following Dirichlet problem
\be\label{a29}
\left\{\begin{array}{ll}
\Delta \phi+\kappa_2^2\phi=0& \qquad {\rm in}\quad D,\\ [2mm]
\phi=G_0(\cdot,z) &\qquad {\rm on}\quad\pa D.\\
\end{array}\right.
\en
Under the assumption that $\kappa_2^2$ is not a Dirichlet eigenvalue for $-\Delta$ in $D$,
Problem (\ref{a29}) has a unique solution $\phi\in H^1(D)$. This leads to the result
that $h_2=G_0(\cdot,z)|_{\pa D}\in Y$.
Finally, it follows from (\ref{a27}) that $G_0(x, z)|_{\G_{b,a}}\in {\rm Range}(L)$
if $z\in D\cup B_1\cup B_2$. This ends the proof.
\end{proof}

We conclude this section with the investigation of the asymptotic behavior of $(v_j, w_j)$ as $z$
approaches the boundary $\pa B_j\cap\G$ from interior of $B_j (j=1,2)$.
For $z^*\in\pa B_j\cap \G$, we choose $\delta_1>0$ to be small enough such that
\ben
z_n:=z^*-\frac{\delta_1}{n}\nu(z^*)\in B_j,\qquad j=1,2,
\enn
for all $n\in\N$. Let $(v_{j,z_n},w_{j,z_n})\in L_{\Delta}^2(B_j)\times L_{\Delta}^2(B_j)$ be the 
solution to the ITP (\ref{b6}) with $z=z_n$. Then we have the following lemma.

\begin{lemma}\label{lemma3}
If $\kappa_1^2$ is neither a transmission eigenvalue of $B_1$ nor that of $B_2$, then we have
\ben\label{a30}
\lim\limits_{n\rightarrow +\infty}\|w_{j,z_n}\|_{L^2(B_j)}=+\infty,\quad\;\;
\lim\limits_{n\rightarrow +\infty}\|v_{j,z_n}\|_{L^2(B_j)}=+\infty
\enn
for $j=1,2$.
\end{lemma}

\begin{proof}
We only prove the first equality for $j=1$. The other cases can proved similarly.

Assume on the contrary that $\|w_{j,z_n}\|_{L^2}$ are uniformly bounded for all $n\in\N$.
Then there exists a positive constant $C>0$, independent of $n$, such that
\be\label{a31}
\|w_{1,z_n}\|_{L^2(B_1)}\leq C.
\en
Define $\widetilde{v}_{1,z_n}:=v_{1,z_n}-G_0(\cdot, y_n)$ with $y_n:=z^*+\frac{\delta_1}{n}\nu(z^*)$ for
sufficiently small $\delta_1>0$.  A direct calculation implies that $(\widetilde{v}_{1,z_n}, w_{1,z_n})$
solves the ITP
\be\label{a32}
\left\{\begin{array}{llll}
\Delta \widetilde{v}_{1,z_n}+\kappa_1^2\widetilde{v}_{1,z_n}=0 & \qquad {\rm in}\; B_1,\\ [1mm]
\Delta w_{1,z_n}+\kappa_2^2w_{1,z_n}=0& \qquad {\rm in}\;B_1,\\ [1mm]
w_{1,z_n}-\widetilde{v}_{1,z_n}=g_{1,n} &\qquad {\rm on}\;\pa B_1,\\ [1mm]
\frac{\pa w_{1,z_n}}{\pa\nu}-\frac{\pa \widetilde{v}_{1,z_n}}{\pa \nu}=g_{2,n}&\qquad{\rm on}\; \pa B_1,
\end{array}\right.
\en
where
\begin{eqnarray*}\label{a33}
g_{1,n}:=G_0(\cdot,z_n)+G_0(\cdot, y_n) \quad{\rm and}\quad g_{2,n}:=\frac{\pa G_0(\cdot,z_n)}{\pa\nu}
+\frac{\pa G_0(\cdot, y_n)}{\pa\nu}.
\end{eqnarray*}
Define $\widetilde{w}:=w_{1,z_n}-\widetilde{v}_{1,z_n}$. It is easily verified by (\ref{a32}) that
$\widetilde{w}$ satisfies
\be\label{a34}
\left\{\begin{array}{ll}
 \Delta \widetilde{w}+\kappa_1^2\widetilde{w}=\eta w_{1,z_n}& \qquad {\rm in}\; B_1,\\ [2mm]
\frac{\pa \widetilde{w}}{\pa\nu}+{\rm i}\widetilde{w}=g_{2,n}+{\rm i}g_{1,n}&\qquad{\rm on}\; \pa B_1.
\end{array}\right.
\en
Standard discussions show that there exists a unique solution $\widetilde{w}\in H^1(B_1)$ to Problem(\ref{a34})
such that
\be\label{a35}
\|\widetilde{w}\|_{H^1(B_1)}\leq C\left(\|w_{1,z_n}\|_{L^2(B_1)}+\|g_{1,n}\|_{H^{-1/2}(\partial B_1)}+\|g_{2,n}\|_{H^{-1/2}(\partial B_1)}\right).
\en
In view of (\ref{a31}), we conclude $\|w_{1,z_n}\|_{L^2(B_1)}<C$ uniformly for $n\in\N$.
Now we are at a position to prove $\|g_{1,n}\|_{H^{-1/2}(\partial B_1)}<C$ uniformly for $n\in\N$.

Noting that
\ben
g_{1,n}(x)&=&G_0(x,z_n)+G_0(x, y_n)\\
&=&\Phi_{\kappa_1}(x,z_n)+\Phi_{\kappa_1}(x,y_n)+\Psi^{(1)}(x,z_n)+\Psi^{(1)}(x,y_n)
\enn
and $\Psi^{(1)}(x,\hat{z})\in C^{\infty}(\R^2\times\R^2\se\{\hat{z}\}\times\{\hat{z}\})$ for $\hat{z}=z_n$
or $\hat{z}=y_n$, then it is enough to show that
$\tilde{g}_{1,n}:=\Phi_{\kappa_1}(x,z_n)+\Phi_{\kappa_1}(x,y_n)
=(\Phi_{\kappa_1}(x,z_n)-\Phi_{\kappa_1}(x,y_n))+2\Phi_{\kappa_1}(x,y_n)$
are uniformly bounded in the $H^{-1/2}$-norm for all $n\in\N$. It follows from the Taylor expansion
for $\Phi_{\kappa_1}(\cdot,y_n)$ and $\Phi_{\kappa_1}(\cdot,z_n)$ that
\ben
\Phi_{\kappa_1}(z^*,z_n)-\Phi_{\kappa_1}(z^*,y_n)=O(t^2\ln\frac{1}{t})
\enn
as $t:=|y_n-z^*|\to 0$. Hence, $\Phi_{\kappa_1}(x,z_n)-\Phi_{\kappa_1}(x,y_n)$ is continuous at $x=z^*$.
It remains to show that $\|\Phi_{\kappa_1}(x,y_n)\|_{H^{-1/2}(\partial B_1)}<C$ uniformly for all $n\in\N$.

Since $y_n\notin B_1$, it is found that $\Phi_{\kappa_1}(\cdot,y_n)$ solves the Helmholtz equation
$\Delta\Phi_{\kappa_1} +\kappa^2_1\Phi_{\kappa_1}=0$ in $B_1$, which implies that $\Phi_{\kappa_1}(\cdot,y_n)$
are uniformly bounded in the norm of $L^2_\Delta(B_1)$ for all $n\in\N$. Using the trace theorem leads to
the result that $\|\Phi_{\kappa_1}(x,y_n)\|_{H^{-1/2}(\partial B_1)}<C$ uniformly for all $n\in\N$.

For the third term in (\ref{a35}), it follows from the proof of \cite[Lemma 4.2]{CKM07} and
the bounded embedding of $L_\infty$ into $H^{-1/2}$ that
\ben
\|g_{2,n}\|_{H^{-1/2}(\pa B_1)}\leq C\|g_{2,n}\|_{L_\infty(\pa B_1)}\le C
\enn
uniformly for $n\in\N$. Combining the above inequalities with (\ref{a31}) and (\ref{a35}) and using
the trace theorem conclude that
\ben
\|\widetilde{w}\|_{H^{1/2}(\pa B_1)}\leq C\|\widetilde{w}\|_{H^1(B_1)}\leq C
\enn
for a fixed constant $C>0$, whence the uniform boundedness of $\|g_{1,n}\|_{H^{1/2}(\partial B_1)}$
follows from the transmission condition of (\ref{a32}). This is a contradiction since
\ben
\|g_{1,n}\|_{H^{1/2}(\pa B_1)}=\|G_0(\cdot,z_n)+G_0(\cdot,y_n)\|_{H^{1/2}(\pa B_1)}\to\infty\quad{\rm as\;}
n\to\infty.
\enn
The proof is thus complete.
\end{proof}

\section{The linear sampling method}\label{sec4}
\setcounter{equation}{0}

Based on the above analysis for Problem (\ref{a13}),  the objective of this section is to propose a 
sampling-type method to simultaneously reconstruct the local perturbation of the interface $\G$ and the 
buried obstacle $D$ from the wave-field measurements $u^{\rm s}(\cdot,y)|_{\G_{b,a}}$ generated by the 
incident source $u^{\rm i}(\cdot)=\Phi_{\kappa_1}(\cdot, y)$ for $y\in\G_{b,a}$.
Then we have the following the near-field operator $N: L^2(\G_{b,a})\rightarrow L^2(\G_{b,a})$ in the form
\ben\label{b2}
Ng(x): = \int_{\G_{b,a}}(u^{\rm s}(x,y)-G_0^{\rm s}(x,y))g(y){\rm d}s(y) \qquad {\rm for}\;\;x\in\G_{b,a},
\enn
where $G_0^{\rm s}(x,y)=G_0(x,y)-\Phi_{\kappa_1}(x,y)$ is the scattered field to Problem (\ref{a9}) associated 
with the interface $\G_0$ and the incident field $\Phi_{\kappa_1}(\cdot, y)$ located at $y\in\G_{b,a}$. 
It is observed that the kernel of $N$ is the unique solution of Problem (\ref{a13}).
Then our sampling method will be based on studying the solvability of the following integral equation of the 
first kind
\be\label{b3}
Ng(\cdot)=G_0(\cdot,z)\qquad {\rm for}\; z\in\R^2,
\en
where $z$ is a sample point belonging to a rectangular domain which contains local perturbations of $\G$ and $D$.
Similar to the bounded obstacle case, it is expected to define an indicator function $I(z)$ by the $L^2$-norm of the solution to equation (\ref{b3}), which can be used to recover all local perturbations $B$ and the buried obstacle $D$.

By the superposition principle, it is known that $N$ corresponds to the incidence operator $H:=(H_1, H_2): L^2(\G_{b,a})\rightarrow X\times Y$, defined by
\ben\label{b4}
&&H_1g(x):=\eta(\chi_1(x)-\chi_2(x))\int_{\G_{b,a}}G_0(x,y)g(y){\rm d}s(y),\;\;{\rm for}\;\;x\in B_1\cup B_2,\\\label{b5}
&&H_2g(x):=-\int_{\G_{b,a}}G_0(x,y)g(y){\rm d}s(y),\qquad\qquad\qquad\quad{\rm for}\;\; x\in\pa D.
\enn
Therefore, $N=LH$. Furthermore, we have the  following denseness result related to $N$.

\begin{lemma}\label{lemma4}
If $\kappa_2^2$ is not a Dirichlet eigenvalue of $-\Delta$ in $D$, then ${\rm Range}(N)$ and ${\rm Range}(L)$
are dense in $L^2(\G_{b,a})$.
\end{lemma}

\begin{proof}
To prove the lemma it is enough to show that the adjoint operator $N^*$ of $N$, given by
\begin{eqnarray}\label{b6}
(N^*\varphi)(y)=\int_{\G_{b,a}}(\ov{u(y,x)}-\ov{G_0(y,x)})\varphi(x){\rm d}s(x)
\quad{\rm for}\;\;\varphi\in L^2(\Gamma_{b,a}),
\end{eqnarray}
is injective on $L^2(\G_{b,a})$. Note that the symmetric relations $u(y,x)=u(x,y)$ and $G_0(y,x)=G_0(x,y)$
for $x,y\in\G_{b,a}$ have been used in deriving (\ref{b6}), which can be obtained by a similar argument
as in \cite{LYZ18}.

By (\ref{b6}) we can define the functions $w_1$ and $w_2$ by
\begin{eqnarray*}\label{b8}
w_1(y)&=&\int_{\G_{b,a}}(u(y,x)-G_0(y,x))\overline{\varphi(x)}{\rm d}s(x)\quad{\rm for }\; x\in \R^2\se\overline{D},\\\label{b9}
w_2(y)&=&-\int_{\G_{b,a}}G_0(y,x)\overline{\varphi(x)}{\rm d}s(x)\qquad\quad\qquad{\rm for }\; x\in \R^2.
\end{eqnarray*}
Let $N^*\varphi = 0$ on $\G_{b,a}$ for some $\varphi\in L^2(\G_{b,a})$.
It then holds that $w_1=0$ on $\G_{b,a}$. An analogous argument with the proof of Lemma \ref{lemma1}
shows that $w_1=0$ in $\R^2\se{(\overline{B_1}\cup\overline{B_2}\cup\overline{D})}$. Thus
we have $w_1=0$ on $\pa D$, and so $w_2=0$ on $\pa D$ by the fact that $u(y,x)=0$ for $y\in\pa D$
and $x\in\G_{b,a}$. Since $\kappa_2^2$ is not a Dirichlet eigenvalue of $-\Delta$ in $D$, it is deduced
that $w_2 =0$ in $D$. Hence, we have $w_2=0$ in $\R^2\se\ov{\G_{b,a}}$ by analytic continuation.
Using the jump relation of the layer operator yields $\varphi=0$, which means that
$N^*$ is injective on $L^2(\G_{b,a})$, implying that $Range(N)$ is dense in $L^2(\G_{b,a})$.

Next we show the denseness of ${\rm Range}(L)$ in $L^2(\G_{b,a})$ by contradiction.
Assume on the contrary that there exists some $g_0\in L^2(\G_{b,a})$ and $\epsilon_0>0$ such that
\begin{eqnarray*}\label{b10}
\|L(h_1,h_2)-g_0\|_{L^2(\G_{b,a})}\geq\epsilon_0
\end{eqnarray*}
for all $(h_1,h_2)\in X\times Y$. Recalling $Hg\in X\times Y$ for $g\in L^2(\G_{b,a})$ and $N=LH$, we thus obtain
\begin{eqnarray*}\label{b11}
\|Ng-g_0\|_{L^2(\G_{b,a})}=\|LHg-g_0\|_{L^2(\G_{b,a})}\geq \epsilon_0.
\end{eqnarray*}
This contradicts the fact that $Range(N)$ is dense in $L^2(\G_{b,a})$. The proof is thus complete.
\end{proof}

\begin{theorem}\label{theorem2}
If $\kappa_2^2$ is not a Dirichlet eigenvalue for $-\Delta$ in $D$, then ${\rm Range}(H)$ is dense in $X\times Y$.
\end{theorem}

\begin{proof}
For any fixed  $(\mu,\xi)\in X\times Y$ with $\mu:=\chi_1\mu_1+\chi_2\mu_2$ for $\mu_{1}\in X_1$ and 
$\mu_2\in X_2$, it is sufficient to show that $\forall \vep>0$, there exists
$g\in L^2(\G_{b,a})$ such that $\|Hg-(\mu,\xi)\|_{X\times Y}<\vep$. That is,
\ben\label{b12}
||\eta p-\mu_{1}||_{L^2(B_1)}+||-\eta p-\mu_2||_{L^2(B_2)}+||-p-\xi||_{H^{1/2}(\pa D)}<\varepsilon,
\enn
where
\ben\label{b13}
p(x):=\int_{\G_{b,a}}G_0(x,y)g(y){\rm d}s(y)\qquad {\rm for}\; x\in\R^2.
\enn
Since $\mu_{1}\in X_1$ and $\mu_2\in X_2$, it follows from the trace theorem that
\ben
\frac{\partial\mu_1}{\partial\nu}+{\rm i}\mu_1\in H^{-3/2}(\partial B_1),\quad\quad 
\frac{\partial\mu_2}{\partial\nu}+{\rm i}\mu_2\in H^{-3/2}(\partial B_2).
\enn
Therefore, it suffices for us to show that
\ben
\{((\pa_{\nu} p+{\rm i}p)|_{\pa B_1},(\pa_{\nu} p+{\rm i}p)|_{\pa B_2}, p|_{\pa D}): g\in L^2(\G_{b,a})\}
\enn
is dense in $H^{-3/2}(\pa B_{1})\times H^{-3/2}(\pa B_{2})\times H^{1/2}(\pa D)$, based on the well-posedness 
of the impedance problem in $B_j$ and the  Dirichlet problem in $D$.
To this end, let $(\phi_1,\phi_2,\phi)\in H^{3/2}(\pa B_1)\times H^{3/2}(\pa B_2)\times H^{-1/2}(\pa D)$ be 
chosen such that
\ben\label{b14}
\sum_{j=1}^{2}\int_{\pa B_j}\left(\frac{\pa p}{\pa\nu}(x)+{\rm i}p(x)\right)\ov{\phi}_j(x){\rm d}s(x)
+\int_{\pa D}p(x)\ov{\phi}(x){\rm d}s(x)=0
\enn
for all $g\in L^2(\G_{b,a})$.

By interchanging the order of integration, it is concluded that
\ben\label{b15}
\int_{\G_{b,a}}g(y)\rho(y){\rm d}s(y)=0 \qquad {\rm for\;}g\in L^2(\G_{b,a}),
\enn
with $\rho$ given by
\ben\label{b16}
\rho(y)=\sum_{j=1}^{2}\int_{\pa B_j}\left(\frac{\pa G_0(x,y)}{\pa\nu(x)}
+{\rm i}G_0(x,y)\right)\ov{\phi_j}(x){\rm d}s(x)+\int_{\pa D}G_0(x,y)\ov{\phi}(x){\rm d}s(x).
\enn
By (\ref{b15}) it is deduced that $\rho=0$ on $\G_{b,a}$. Then, and since $\rho$ is analytic on $\G_b$,
it follows by analytic continuation and the uniqueness result of the Dirichlet problem in $U_b$
that $\rho=0$ in $\R^2\se\{\ov{D}\cup\ov{B_1}\cup\ov{B_{2}}\}$.
Use the jump relations of the layer potentials to obtain that $[\rho]=0$ on $\pa D$ and
\ben
[\rho]=\ov{\phi_j},\quad [\pa_\nu\rho]=-{\rm i}\ov{\phi_j}\qquad {\rm on\;} \partial B_j
\enn
for $j=1,2$, where $[\cdot]$ indicates the difference of the limits of the function approaching
the boundary from the exterior and interior domains of $B_j$ and $D$, respectively.
Note that $\rho$ is the unique solution of the impedance problem
\ben\label{b18}
\Delta\rho+\kappa^2_1\rho=0\quad{\rm in}\; B_1,\qquad \pa_\nu\rho+{\rm i}\rho=0\quad{\rm on}\;\pa B_1
\enn
and the Dirichlet problem
\ben\label{b19}
\Delta\rho+\kappa_2^2\rho=0\quad {\rm in}\; D,\qquad \rho=0\quad{\rm on}\;\;\pa D.
\enn
Thus, $\rho=0$ in $B_1\cup D$ since $\kappa^2_2$ is not a Dirichlet eigenvalue for $-\Delta$ in $D$,
leading to the fact that $\phi_1=0$ and $\phi=0$.
Similarly, we have $\phi_2=0$. The proof is thus complete.
\end{proof}

With the above analysis, we are ready to present the sampling method for simultaneously
reconstructing the shape and location of the rough interface and the buried obstacle by equation (\ref{b3}).

\begin{theorem}\label{theorem3}
If $\kappa_1^2$ is neither a transmission eigenvalue of $B_1$ nor that of $B_2$ and $\kappa_2^2$ is not
an Dirichlet eigenvalue of $-\Delta$ in $D$, then the following statements hold. 

$(1)$ For $z\in B_1\cup B_2\cup D$ and $\vep>0$, there exists $g_{z,\vep}\in L^2(\G_{b,a})$
satisfying the inequality
\be\label{b30}
\|Ng_{z,\vep}(x)-G_0(x,z)\|_{L^2(\G_{b,a})}<\vep
\en
such that $||g_{z,\vep}||_{L^2(\G_{b,a})}\rightarrow\infty$ and $||Hg_{z,\vep}||_{X\times Y}\rightarrow\infty$
as $z$ approaches $(\G\se\G_0)\cup\pa D$.

$(2)$ For $z\notin B_1\cup B_2\cup D$ and $\vep>0$ and $\delta>0$, there exists a
$g_{z,\vep}^{\delta}\in L^2(\G_{b,a})$ satisfying the inequality
\be\label{b31}
\|Ng^{\delta}_{z,\vep}(x)-G_0(x,z)\|_{L^2(\G_{b,a})}<\vep+\delta,
\en
such that $||g^{\delta}_{z,\vep}||_{L^2(\G_{b,a})}\rightarrow\infty$ and
$||Hg_{z,\vep}||_{X\times Y}\rightarrow\infty$ as $\delta\rightarrow 0$.
\end{theorem}

\begin{proof}
We first assume that $z\in B_1\cup B_2\cup D$.  Since $(v_j, w_j) (j=1,2)$ are the solution
to the ITP (\ref{a26}), we define $h_1:=\eta(\chi_1v_1+\chi_2w_2)$ and $h_2:=G_0(\cdot,z)|_{\pa D}$.
Then it holds $(h_1,h_2)\in X\times Y$ and by Lemma \ref{lemma2},
\be\label{b20}
L(h_1,h_2)=G_0(\cdot,z)|_{\G_{b,a}}.
\en
Under the assumption on $\kappa^2_2$, it follows from Theorem \ref{theorem2} that ${\rm Range}(H)$
is dense in $X\times Y$. So, for any $\vep>0$ there exists a function $g_{z,\vep}\in L^2(\G_{b,a})$
such that
\be\label{b21}
\|Hg_{z,\vep}-(h_1,h_2)\|_{X\times Y}<\vep
\en
which gives
\ben\label{b22}
\|Ng_{z,\vep}(x)-G_0(x,z)\|_{L^2(\G_{b,a})}\lesssim\vep
\enn
by combining (\ref{b20})-(\ref{b21}) and the boundedness of $L$.

Next, it remains to show $\|g_{z,\vep}\|_{L^2(\G_{b,a})}\rightarrow\infty$ and
$||Hg_{z,\vep}||_{X\times Y}\rightarrow \infty$ as $z$ approaches $(\G\se\G_0)\cup\pa D$.
Assume on the contrary that there exists a fixed constant $C>0$ such that
$\|g_{z,\vep}\|_{L^2(\G_{b,a})}\leq C$. Then, and from the boundedness of $H$ it follows that
$\|Hg_{z,\vep}\|_{X\times Y}\lesssim C$. We thus have
\ben
\|h_1\|_{X}+\|h_2\|_{Y}\lesssim C+\varepsilon
\enn
by (\ref{b21}), which means that
\begin{equation}\label{b23}
\|v_1\|_{L^2(B_1)}+\|w_2\|_{L^2(B_2)}+\|G_0(\cdot,z)\|_{H^{1/2}(\pa D)}\lesssim C+\varepsilon
\end{equation}
from the definition of $h_1$ and $h_2$.
For the case $z\to \G\se\G_0$, it follows from Lemma \ref{lemma3} that $\|v_1\|_{L^2(B_1)}+\|w_2\|_{L^2(B_2)}\to+\infty$ which contradicts with inequality (\ref{b23}). For the case $z\to \pa D$, it is easily checked that $\|G_0(\cdot,z)\|_{H^{1/2}(\pa D)}\rightarrow \infty$ which also contradicts with inequality (\ref{b23}). Hence, it is deduced that $\|g_{z,\varepsilon}\|_{L^2(\G_{b,a})}\rightarrow \infty$ and $\|Hg_{z,\varepsilon}\|_{X\times Y}\rightarrow \infty$ as $z\to(\G\se\G_0)\cup\pa D$.

Next we consider the case $z\notin B_1\cup B_2\cup D$. By Lemma \ref{lemma2}, it is known that $G_0(\cdot,z)|_{\G_{b,a}}$ is not in the range of $L$. Thus, it is not solvable for the first kind of operator equation $Lh=G_0(\cdot,z)|_{\G_{b,a}}$. However, by Lemma \ref{lemma1} we have that $L:X\times Y\to L^2(\G_{b,a})$ is compact and injective. Using \cite[Theorem 4.13 ]{CK13}, it can be shown that the regularized equation
\begin{eqnarray*}\label{b24}
\alpha h_{\alpha}+L^*Lh_{\alpha}=L^*(G_0(\cdot,z)|_{\G_{b,a}})
\end{eqnarray*}
always has a unique solution $h_{\alpha}=(h_{1\alpha}, h_{2\alpha})\in X\times Y$ for each regularized parameter $\alpha>0$, which can be represented as
\begin{eqnarray*}\label{b25}
h_{\alpha}=\sum_{n=1}^{\infty}\frac{\lambda_n}{\alpha+\lambda_n^2}(G_0(\cdot,z),\psi_n)\varphi_n.
\end{eqnarray*}
Here, $(\lambda_n,\varphi_n,\psi_n)$ denotes a singular system of the operator $L$. Since ${\rm Range}(L)$ is dense on $L^2(\G_{b,a})$, we apply the Picard theorem (cf. \cite{CK13}) to deduce
\begin{eqnarray}\label{b26}
\lim_{\alpha\to 0}\|h_{\alpha}\|_{X\times Y}=\infty.
\end{eqnarray}
The standard discussion now shows that
for $0<\delta<\|G_0(\cdot,z)\|_{L^2(\G_{b,a})}$, there exists a unique parameter $\alpha$ satisfying the equation
\begin{eqnarray}\label{b27}
\|Lh_{\alpha}-G_0(\cdot,z)\|_{L^2(\G_{b,a})}=\delta.
\end{eqnarray}
Recalling $h_{\alpha}\in X\times Y$, we conclude by the denseness of ${\rm Range}(H)$ that there exists $g_{z,\varepsilon}^{\delta}\in L^2(\G_{b,a})$ such that
\begin{eqnarray}\label{b28}
\|LHg_{z,\varepsilon}^{\delta}-Lh_{\alpha}\|<\varepsilon
\end{eqnarray}
for any given  $\varepsilon>0$.

Finally, by combining (\ref{b27}) and (\ref{b28}), we arrive at
\ben\no
\|Ng_{z,\varepsilon}^{\delta}-G_0(\cdot,z)\|_{L^2(\G_{b,a})}
&=&\|LHg_{z,\varepsilon}^{\delta}-G_0(\cdot,z)\|_{L^2(\G_{b,a})}\\\no
&\leq&\|LHg_{z,\varepsilon}^{\delta}-Lh_{\alpha}\|_{L^2(\G_{b,a})}+\|Lh_{\alpha}-G_0(\cdot,z)\|_{L^2(\G_{b,a})}\\\label{b29}
&<&\varepsilon+\delta.
\enn
Noticing $\alpha\to 0$ as $\delta\to 0$, it can be checked by (\ref{b26}) and (\ref{b28}) that $\|g^{\delta}_{z,\varepsilon}\|_{L^2(\G_{b,a})}$ $\to\infty$ and $\|Hg_{z,\varepsilon}^{\delta}\|_{X\times Y}\rightarrow\infty$ as $\delta\rightarrow 0$.
The proof is thus complete.
\end{proof}

By Theorem \ref{theorem3}, it is found that the solution $g_z$ of equation (\ref{b3}) in the sense of inequalities (\ref{b30}) and (\ref{b31}) has totally different behaviors when the sampling point $z$ lies inside or outside of the domain $B_1\cup B_2\cup D$, which provides a qualitative way to visualize the local perturbation $B_1\cup B_2$ and the embedded obstacle $D$.  Based on this observation, we define the indicator function
\ben
{\rm Ind}(z):=1/\|g_z\|_{L^2(\G_{b,a})}
\enn
where $g_z$ is the solution of equation (\ref{b30}) and (\ref{b31}). It follows from the Theorem \ref{theorem3} 
again that ${\rm Ind}(z)$ is small when the sampling point $z$ approaches the local perturbation $\G\se\G_0$ 
or $z$ approaches $\partial D$ from inside of $B_1\cup B_2\cup D$. Therefore, ${\rm Ind}(z)$ can provide 
a fast imaging algorithm. The following procedure shows how to numerically reconstruct the shape and location 
of $B_1$, $B_2$ and $D$ by ${\rm Ind}(z)$.\\

\begin{algorithm}\caption{Reconstruction of locally rough interfaces and buried obstacles by the LSM}\label{alg1}

\begin{itemize}

\item Select a rectangular grid $S$ containing the local perturbation of the scattering interface $\G$ and the buried obstacle $D$;

\item Solve the scattering problem (\ref{a4})-(\ref{a8}) and (\ref{a9}) to obtain the wave-field data $u(x,y)$ and $G_0(x,y)$ for $x,y\in \G_{b,a}$ by the Nystr\"{o}m method. Then, solve Problem (\ref{a9}) again to obtain the data $G_0(x,z)$ for each sampling point $z\in S$;

\item For each $z\in S$, solve the near-field equation (\ref{b3}) to obtain an approximate solution $g_z$, based on the Tikhonnov regularization with the Morozov discrepancy principle; 

\item Choose a cut-off value $C>0$ and compute the indicator function ${\rm Ind}(z)$ so that it is in practice reasonable to detect $z\in B_1\cup B_2\cup D$ if and only if ${\rm Ind}(z)\leq C$.

\end{itemize}
\end{algorithm}

\section{Numerical results}\label{sec5}
\setcounter{equation}{0}

Following Algorithm \ref{alg1}, some numerical examples are carried out to demonstrate the performance of the 
sampling method in Theorem \ref{theorem3}, based on numerically solving the following equation
\be\label{d1}
Ng(x)=G_0(x,z)\qquad {\rm for}\; x\in\G_{b,a}.
\en
Recall that the kernel $u(x,y)-G_0(x,y)=u^s(x,y)-G_0^s(x,y)$ is analytic, leading to that equation (\ref{d1}) 
is severely ill-posed. Therefore, equation (\ref{d1}) has to be solved by considering its regularized equation
\be\label{d2}
\alpha g_z^{\alpha}+N^*N g_z^{\alpha}=N^*(G_0(x,z))|_{\G_{b,a}}
\en
with the regularization parameter $\alpha$ chosen by the Morozov discrepancy principle.

In next numerical examples, the synthetic data $u^s(x,y)$, $G^s_0(x,y)$ and $G_0(x,z)$ are obtained by 
solving the scattering problems (\ref{a4})-(\ref{a8}) and (\ref{a9}) with the Nystr\"{o}m method 
(cf. \cite{LYZ13}). Then the near-field operator $N$ can be discretized into the following finite 
dimensional matrix
\ben
N_{n\times n}=(u^s(x_j, y_l)-G_0^s(x_j,y_l))_{1\leq j, l\leq n},
\enn
where $x_j$ is the measuring points equally distributed at $\G_{b,a}$ with $j=1,2,\cdots,n$, and $y_l$ is 
the incident point sources which is also equally distributed at $\G_{b,a}$ with $j=1,2,\cdots,n$. Moreover, 
the test function $G_0(x,z)$ is also discretized as a finite dimensional vector $(G_0(x_j,z))_{1\leq j\leq n}$ 
for each sampling point $z\in S$. Thus we have the following discretization regularized equation
\be\label{d4}
\alpha g_{z}^{\alpha}+N^*_{n\times n}N_{n\times n}g_{z}^{\alpha}=N^*_{n\times n}(G_0(x_j,z))_{1\leq j\leq n}
\en
for equation (\ref{d2}).
We can then define the indicator function
\ben
{\rm Ind}_n(z)= 1{\left/\left(\sum_{1\leq j\leq n}|g^\alpha_{z,j}|^2\right)^{\frac{1}{2}}\right.}.
\enn
in the discrete form for $g^\alpha_z:=(g^\alpha_{z,1},\cdots,g^\alpha_{z,n})^T\in\mathbb{C}^n$.

By Theorem \ref{theorem3}, it can be deduced that ${\rm Ind}_{n}(z)$ should be very small for 
$z\notin B_1\cup B_2\cup D$ and consideraly large for $z\in B_1\cup B_2\cup D$ if ${\rm Ind}_n(z)$
approximates ${\rm Ind}(z)$. Furthermore, in order to present the results under the same standard, 
we normalize ${\rm Ind}_{n}(z)$ to obtain a new indicator function
\ben
{\rm NInd}(z):={\rm Ind}_{n}(z)/\max \limits_{z\in S}{\rm Ind}_{n}(z)
\enn
which will be used in the following numerical examples to reconstruct the shape and location 
of $B_1$, $B_2$ and $D$.

To test the stability of the inversion algorithm, we also consider equation (\ref{d4}) with noisy data. 
In this case, the wave-field data $u^s(x,y)$ is given by
\ben
((u^s(x_j,y_l))_{n\times n})_{\delta}:=(u^s(x_j,y_l))_{n\times n}
+\delta\frac{\zeta}{\|\zeta\|_2}\|u^s(x_j,y_l)_{n\times n}\|_2,
\enn
for relative error $\delta>0$, where $\zeta_{n\times n}=(\zeta_1)_{n\times n}+{\rm i}(\zeta_2)_{n\times n}$ 
is a complex-valued matrix with its real part $\zeta_1$ and imaginary part $\zeta_2$ consisting of 
random numbers obeying standard normal distribution $N(0,1)$. Then Algorithm \ref{alg1} could be reduced to the following form.

\begin{algorithm}\caption{Reconstruction of locally rough interfaces and buried obstacles by the LSM}\label{alg2}
\begin{itemize}

\item Select a rectangular grid $S$ containing the local perturbations of the scattering interface $\G$ and the buried obstacle $D$;

\item Solve Problem (\ref{a4})-(\ref{a8}) and (\ref{a9}) to obtain the synthetic data $(u^s(x_j,y_l))_{n\times n}$, $(G_0^s(x_j,y_l))_{n\times n}$ and $(G_0(x_j,z))_{n\times 1}$ for each $z\in S$ by the Nystr\"{o}m method;

\item Solve the discretization regularized equation (\ref{d4}) for each $z\in S$ to obtain its solution $g^\alpha_z$ with different noisy data level; 

\item Compute the indicator function 
${\rm NInd}(z):={\rm Ind}_{n}(z)/\max \limits_{z\in S}{\rm Ind}_{n}(z),$
and then plot the mapping ${\rm NInd}(z)$ against $z$.

\end{itemize}
\end{algorithm}

Unless otherwise stated, we set the wavenumber $\kappa_1=1$ and $\kappa_2=2$, and consider the sampling points $z$ in the rectangular grid $(-5, 5)\times (-8.5,1.5)$ 
with the step size $0.06$ in $x$-axis and $0.06$ in $y$-axis. The measurement width  and height are chosen to 
be $a=20$ and $b=1.55$, respectively, for $\G_{b,a}$, and the number of measurement points is chosen to 
be $n=401$ which are uniformly distributed on $\G_{b,a}$.
Moreover, lots of numerical examples we have carried out show that $\alpha(z)$ can be taken to be a fixed constant. 
Here, we choose $\alpha(z)=10^{-6}$ in Algorithm 2 for the case of noisy data.

{\bf Example 1.} In this example, numerical results are presented for two particular cases. 
The first one ((a) in Figure \ref{f2}) is related to a planar surface $\G:=\G_0$ and the buried obstacle $D$ described by a circle
\ben
x(\theta) = (0.4\cos(\theta),-3+0.4\sin(\theta)),\quad \theta\in [0, 2\pi).
\enn
The second one (see $(b)$ in Figure \ref{f2}) is just related to a locally rough surface $\G$, described 
by $f(t)=-2\Omega_3(t)$, without buried obstacle in the lower half-space,  where $\Omega_3(\cdot)$ is a 
cubic $B$-spline function which is twice continuously differentiable with compactly support in $\R$ and is 
given by
\ben
\Omega_3(t)
= \left\{\begin{aligned}
\frac{1}{2}|t|^3-t^2+\frac{2}{3}     &\qquad\qquad \textrm{for}\; |t|\leq1,\\
-\frac{1}{6}|t|^3+t^2-2|t|+\frac{4}{3} & \qquad\qquad \textrm{for}\;1<|t|<2\,,\\
 0 &\qquad\qquad\textrm{for}\; |t|\geq2.
\end{aligned}
\right.
\enn
Figure \ref{f2} shows a satisfactory 
reconstruction for the two cases with different noisy level.

{\bf Example 2.} In this example, the rough interface $\G$ and the buried obstacle $D$ are described by 
(see (a) Figure \ref{f3})
\be\label{d5}
&&f(t)=2\Omega_3(2t+7)-2\Omega_3(2t)+2\Omega_3(2t-7),\;\; t\in \R,\\\label{d6}
&&x(\theta)=(-3+\rho(\theta)\cos(\theta),-7+\rho(\theta)\sin(\theta)),\quad \theta\in[0, 2\pi),
\en
where
\ben
\rho(\theta)=\frac{0.5+0.4\cos(\theta)+0.1\sin(2\theta)}{1+0.7\cos(\theta)}.
\enn
The numerical results are demonstrated with the exact data in $(b)$ of Figure \ref{f3}. It is readily seen that our sampling method can acquire a good reconstruction for the rough surface $\G$ and the buried obstacle $D$ especially for $\G$. To improve the performance for $D$, one can try to separately image $\G$ and $D$ 
which depends on some a priori information on $\G$ and $D$. In this case, we can choose two grids $S_1$ 
and $S_2$ with $(B_1\cup B_2)\subset S_1$, $D\subset S_2$ and $S_1\cap S_2=\emptyset$.
Then we implement Algorithm \ref{alg2} for $S_1$ and $S_2$, respectively, with $S_1=(-5,5)\times (-2,1.5)$, 
$S_2 = (-5,5)\times(-8.5,-2)$, which actually improve the reconstruction quality of $D$ and $\G$; see (c) and (d) 
of Figure \ref{f3}.

{\bf Example 3}. In this example, we aim to exam the dependence of our method on the relative position and 
distance between the local perturbation and the buried obstacle. The locally rough interface
$\G$ and the buried obstacle $D$ (rounded square) are described by
\be\label{d7}
&&f(t)=-1.5e^{-3(t-3)^2}\cdot f_0(t)\quad t\in \R,\\\label{d8}
&&x(\theta)=(-3+0.25(\cos^3(\theta)+\cos(\theta)),-3+0.25(\sin^3(\theta)+\sin(\theta))) \;\theta\in [0,2\pi)
\en
where $f_0(t)\in C_0^{\infty}(\R)$ is a cut-off function defined by
\ben
f_0(t)
= \left\{\begin{aligned}
1 &\qquad\qquad \textrm{for}\; |t|\leq4,\\
\;\left(1+e^{\frac{1}{5-|t|}+\frac{1}{4-|t|}}\right)^{-1} & \qquad\qquad \textrm{for}\;4<|t|<5\,,\\
 0 &\qquad\qquad\textrm{for}\; |t|\geq5;
\end{aligned}
\right.
\enn
see (a) in Figure \ref{f4}.

In this example, we demonstrate 
the reconstruction in (c) of Figure \ref{f4} for exact data and (e) for $2\%$ noisy data.
It is seen that the method can give satisfactory reconstructions for $\G$ and $D$, especially for 
the case of exact data. Next we fix $\G$ and then move the buried obstacle $D$ 
to the right by six units so that it lies below the perturbation of $\G$; see $(d)$ and $(f)$ of 
Figure \ref{f4} for the reconstructions, where
the perturbation of $\G$ and $D$ seem not to be distinguished very well in the case of 2\% noisy data. 
We guess that a possible reason is due to the strong multiple scattering between $\G$ and $D$ when the 
distance of $\G$ and $D$ is relatively close. 

{\bf Example 4.} Finally, we consider the rough surface $\G$ with multiple perturbations in three different cases. In the first case, the rough surface $\Gamma$ is described by the curve
\be\label{d9}
f(t)=\left[e^{-3(t+2)^2}+e^{-3(t-2)^2}\right]\cdot f_0(t),\quad t\in\R,
\en
which has two local perturbations. In the second case, the rough surface $\Gamma$ is given by
\be\label{d10}
f(t)=\left[e^{-8(t+4)^2}+e^{-8(t+2)^2}-1.5e^{-6(t-2)^2}+e^{-8(t-4)^2}\right]\cdot f_0(t),\quad t\in\R,
\en
which has four local perturbations. In the last case, the rough surface $\Gamma$ is described by
\be\no
f(t)&=&\Big[e^{-12(t+4)^2}+e^{-12(t+2.5)^2}-2e^{-12(t+1)^2}\\\label{d11}
&&+e^{-10(t-1)^2}+e^{-16(t-2.5)^2}+e^{-12(t-4)^2}\Big]\cdot f_0(t),\quad t\in\R,
\en
which has six local perturbations. And for these three cases, 
 the buried obstacle $D$ is described by an ellipse curve
\be\label{d12}
x(\theta)=(0.6\cos(\theta),-6+0.3\sin(\theta))\quad \theta\in [0,2\pi).
\en
We present the numerical results in Figure \ref{f5}. The first, second, and third row of Figure \ref{f5} is the physical configuration and the reconstruction from exact data for the first, second, and third case, repectively. The results in Figure \ref{f5} shows that our method remains to give a satisfactory reconstruction of multiple perturbations.

\begin{figure*}
\centering
\subfigure[Physical configuration]{\includegraphics[width=0.45\textwidth]{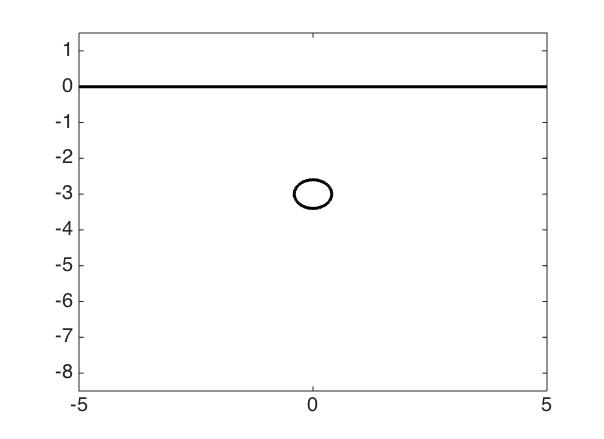}
}
\subfigure[Physical configuration]{\includegraphics[width=0.45\textwidth]{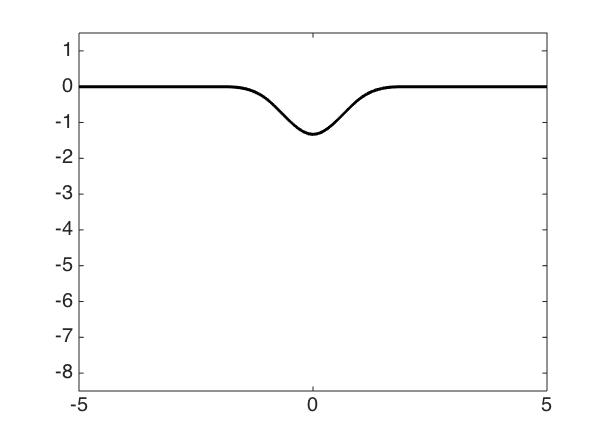}
}
\subfigure[No noise]{\includegraphics[width=0.45\textwidth]{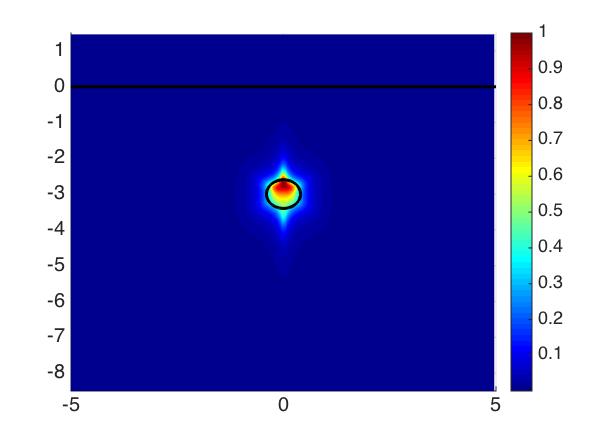}
}
\subfigure[No noise]{\includegraphics[width=0.45\textwidth]{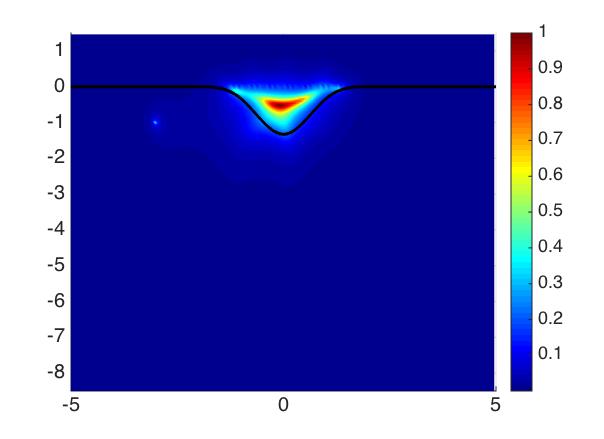}
}
\subfigure[2\% noise]{\includegraphics[width=0.45\textwidth]{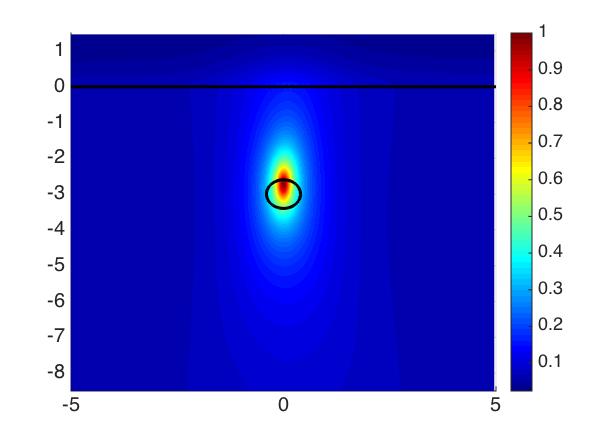}
}
\subfigure[2\% noise]{\includegraphics[width=0.45\textwidth]{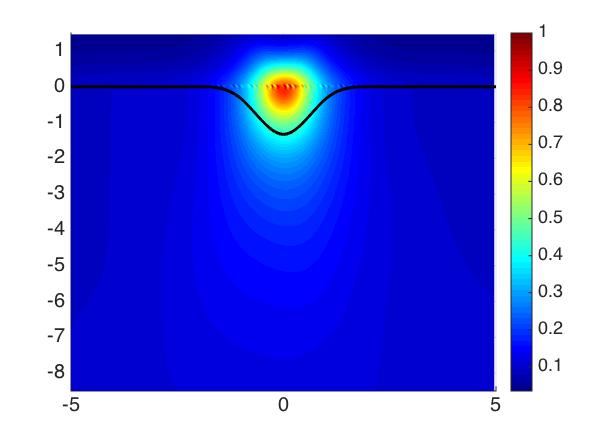}
}
\caption{The left is the reconstruction of a circle-shaped obstacle with a planar interface for the 
exact data and $2\%$ noise, and the right is the reconstruction of a local interface without buried 
obstacles for the exact data and $2\%$ noise. 
}\label{f2} 
\end{figure*}

\begin{figure*}
\centering
\subfigure[Physical configuration]{\includegraphics[width=0.45\textwidth]{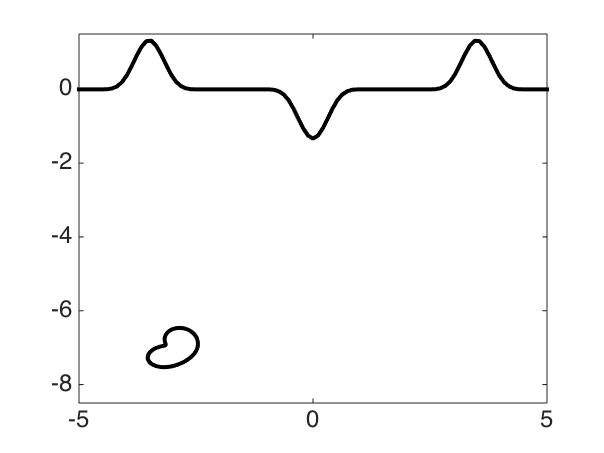}
}
\subfigure[No noise]{\includegraphics[width=0.45\textwidth]{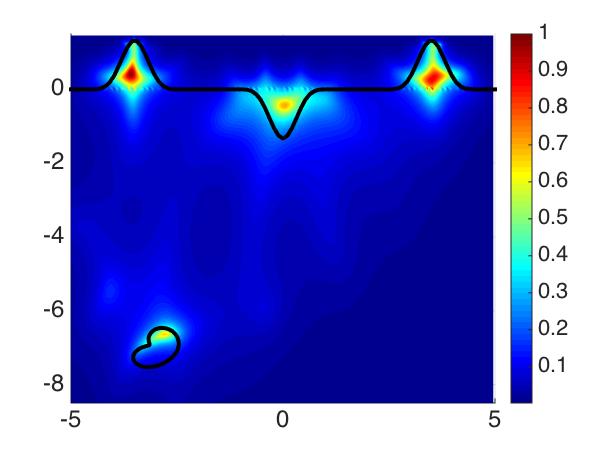}
}
\subfigure[No noise]{\includegraphics[width=0.45\textwidth]{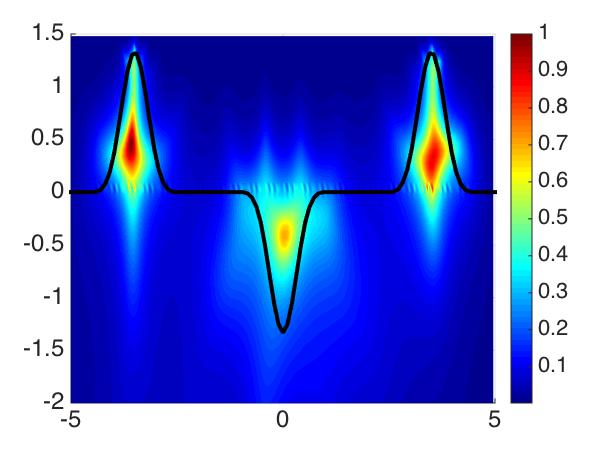}
}
\subfigure[No noise]{\includegraphics[width=0.45\textwidth]{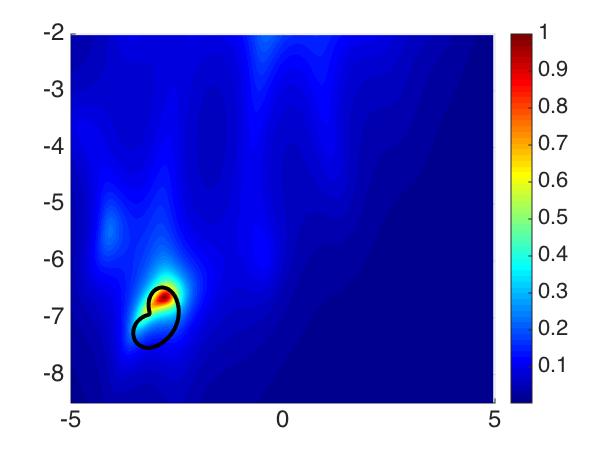}
}
\subfigure[2\% noise]{\includegraphics[width=0.45\textwidth]{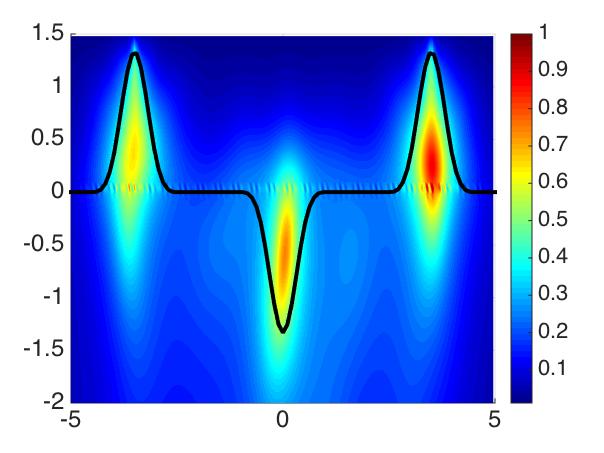}
}
\subfigure[2\% noise]{\includegraphics[width=0.45\textwidth]{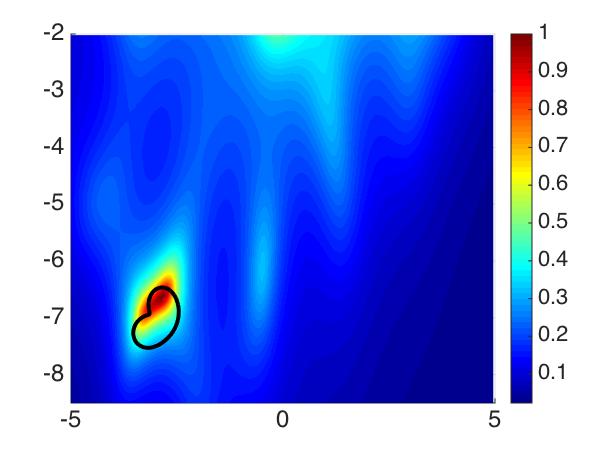}
}
\caption{The reconstructions of $\G$ given by (\ref{d5}) and an apple-shaped obstacle $D$ given by (\ref{d6}) 
. Picture (b) presents the reconstruction for $\G$ and $D$ with no noise data. Pictures (c) and (d) are separately sampling for $\G$ and $D$ with exact data, respectively. Pictures (e) and (f) are separately sampling for $\G$ and $D$ with 2\% noise, respectively. 
}\label{f3} 
\end{figure*}

\begin{figure*}
\centering
\subfigure[Physical configuration]{\includegraphics[width=0.45\textwidth]{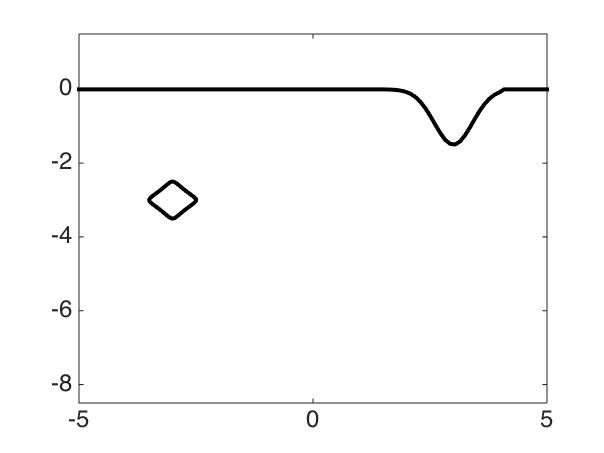}
}
\subfigure[Physical configuration]{\includegraphics[width=0.45\textwidth]{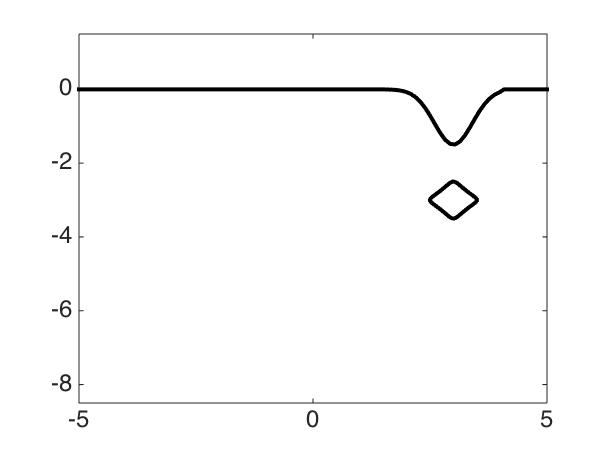}
}
\subfigure[No noise]{\includegraphics[width=0.45\textwidth]{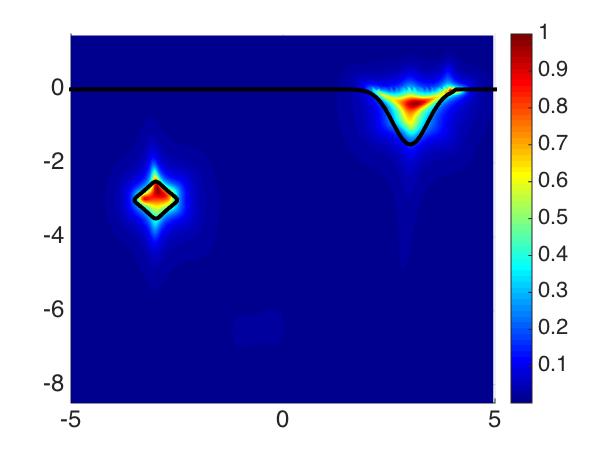}
}
\subfigure[No noise]{\includegraphics[width=0.45\textwidth]{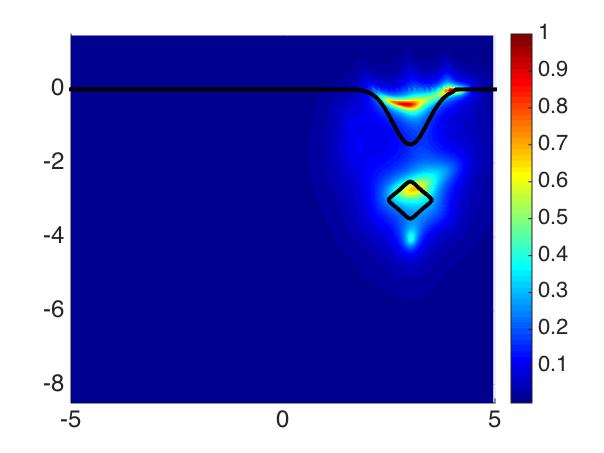}
}
\subfigure[2\% noise]{\includegraphics[width=0.45\textwidth]{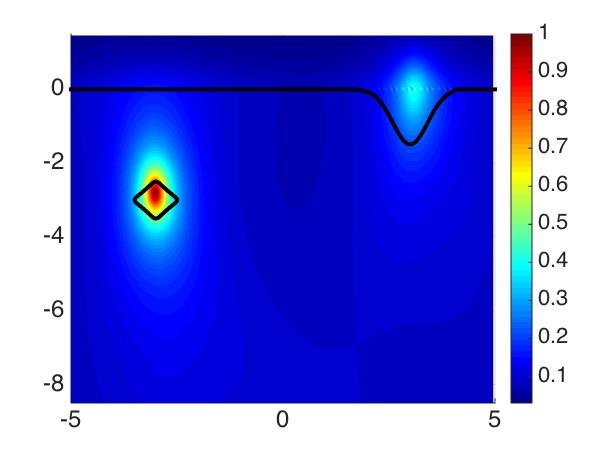}
}
\subfigure[2\% noise]{\includegraphics[width=0.45\textwidth]{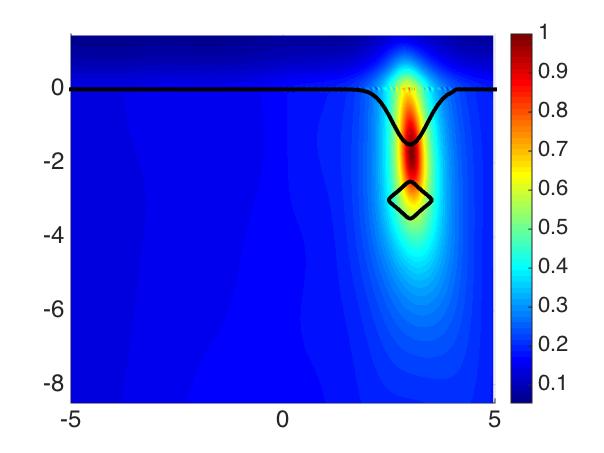}
}
\caption{The reconstructions of $\G$ lying below $\Gamma_0$ and a rounded square obstacle with exact and 2\% noisy data. Pictures (c) and (e) present the reconstructions for both $\G$ and $D$ given 
by (\ref{d7})-(\ref{d8}). Pictures (d) and (f) present the reconstruction for the case where $\G$ fixes 
but $D$ moves to the right by six units. }\label{f4}
\end{figure*}

\begin{figure*}
\centering
\subfigure[Physical configuration]{\includegraphics[width=0.45\textwidth]{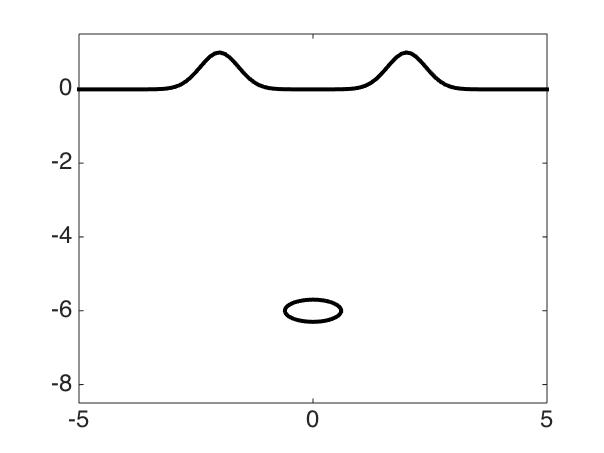}
}
\subfigure[No noise]{\includegraphics[width=0.45\textwidth]{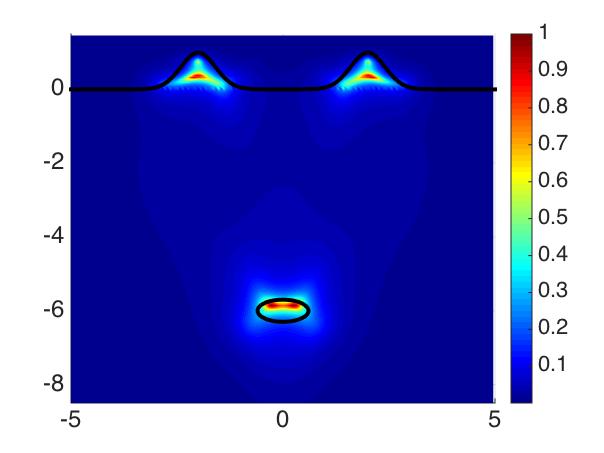}
}
\subfigure[Physical configuration]{\includegraphics[width=0.45\textwidth]{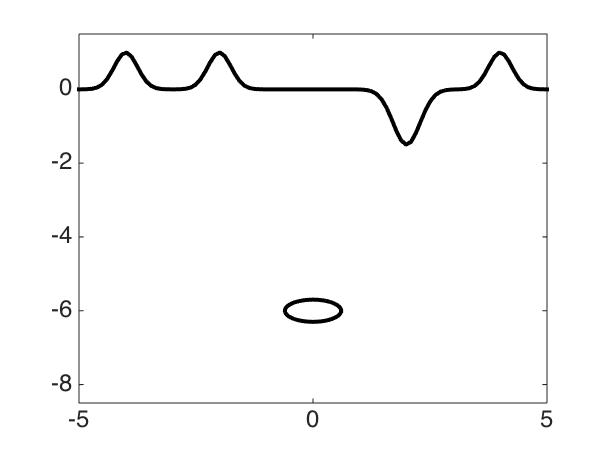}
}
\subfigure[No noise]{\includegraphics[width=0.45\textwidth]{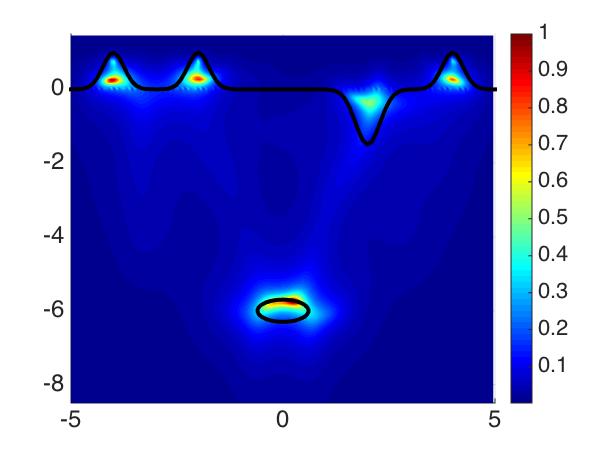}
}
\subfigure[Physical configuration]{\includegraphics[width=0.45\textwidth]{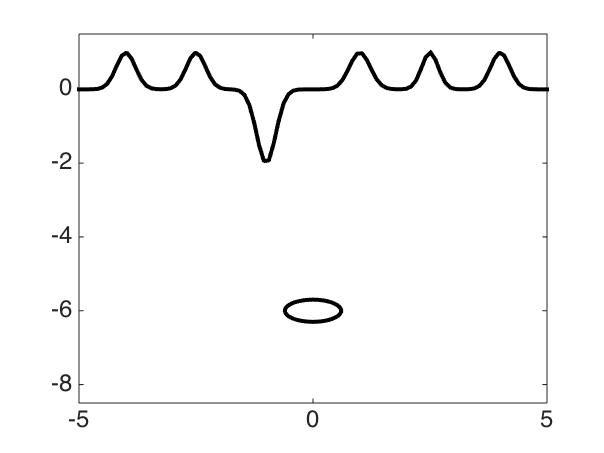}
}
\subfigure[No noise]{\includegraphics[width=0.45\textwidth]{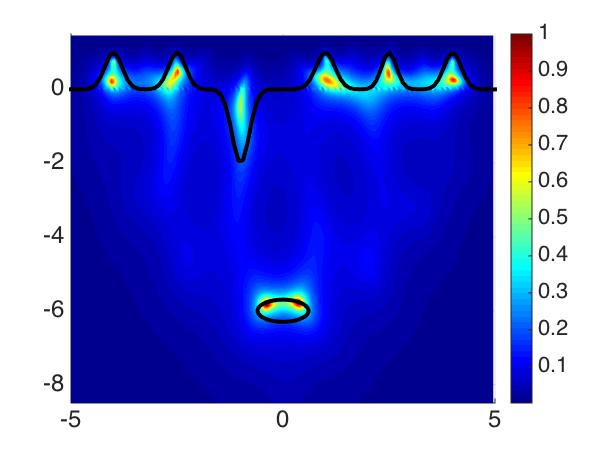}
}
\caption{Reconstructions of three different cases with exact data. Picture (b) presents the reconstruction 
for $\G$ with two local perturbations and an ellipse shape obstacle, given by (\ref{d9}) and (\ref{d12}).Picture (d) presents the reconstruction 
for $\G$ with four local perturbations and an ellipse shape obstacle, given by (\ref{d10}) and (\ref{d12}).Picture (f) presents the reconstruction 
for $\G$ with six local perturbations and an ellipse shape obstacle, given by (\ref{d11}) and (\ref{d12}). }\label{f5} 
\end{figure*}

From the above numerical experiments, it can be observed that the sampling method proposed in 
Theorem \ref{theorem3} can provide satisfactory reconstructions for simultaneously recovering locally 
rough interfaces and the buried obstacles at different noise levels. In addition, it is easily observed 
that the quality of the reconstruction depends on the relative location and distance between the local 
perturbations of the interface and the buried obstacle which possibly corresponds to the different 
strengths of multiple scattering.

\section{Conclusions}\label{sec6}

In this paper, we proposed a novel sampling-type method to simultaneously reconstruct both the local perturbation
of the rough interface and the obstacles buried in the lower half-space from the near-field measurements above
the interface. The idea is mainly based on constructing a modified near-field equation via transferring the 
original scattering problem into the one by an inhomogeneous medium of compact support and the buried obstacles.
Numerical results demonstrated that our inversion algorithm can give satisfactory reconstructions for a variety 
of locally rough interfaces and buried obstacles. Further, the reconstruction can also be regarded as
a good initial guess for an iterative type method in order to obtain an accurate numerical reconstruction 
of the interface and buried obstacles. As far as we know, this is the first sampling-type method to reconstruct 
both the locally rough interface and the buried obstacles simultaneously.  
We remark that it remains open to develop a sampling-type method to recover a nonlocal perturbation of a plane surface. 
We hope to report the progress on this topic in the future.

\section*{Acknowledgments}

This work was supported by the NNSF of China grant 11771349.

\setcounter{equation}{0}


\begin{thebibliography}{99}

\bibitem{AAY07} Y. Altuncu, I. Akduman and A. Yapar, Detecting and locating dielectric objects buried
under a rough interface, {\em IEEE Geosci. Remote Sens. Letters \bf4} (2007), 251-255.

\bibitem{BHY2018} G. Bao, G. Hu and T. Yin, Time-harmonic acoustic scattering from locally
perturbed half-planes, {\em SIAM J. Appl. Math. \bf78} (2018), 2672-2691.

\bibitem{GL13} G. Bao and P. Li, Near-field imaging of infinite rough surfaces,
{\em SIAM J. Appl. Math. \bf73} (2013), 2162-2187.

\bibitem{GL14} G. Bao and P. Li,  Near-field imaging of infinite rough surfaces in dielectric media,
{\em SIAM J. Imaging Sci. \bf7} (2014), 867-899.

\bibitem{GJ11} G. Bao and J. Lin, Imaging of local surface displacement on an infinite ground plane:
the multiple frequency case, {\em SIAM J. Appl. Math. \bf71} (2011), 1733-1752.

\bibitem{CR10} C. Burkard and R. Potthast, A multi-section approach for rough surface reconstruction via
the Kirsch-Kress scheme, {\em Inverse Problems \bf 26} (2010), 045007.

\bibitem{SE10} S. N. Chandler-Wilde and J. Elschner, Variational approach in weighted Sobolev spaces to
scattering by unbounded rough surfaces, {\em SIAM J. Math. Anal. \bf42} (2010), 2554-2580.

\bibitem{CG11} L. Chorfi and P. Gaitan, Reconstruction of the interface between two-layered media
using far-field measurements, {\em Inverse Problems \bf27} (2011) 075001.

\bibitem{CGH10}F. Cakoni, D. Gintides and H. Haddar, The existence of an infinite discrete set of
transmission eigenvalues, {\em SIAM J. Math. Anal. \bf42} (2010), 237-255.

\bibitem{CK13} D. Colton and R. Kress, Inverse Acoustic and Electromagnetic Scattering Theorey (3rd Ed.),
Springer, 2013.

\bibitem{CKM07} D. Colton, R. Kress and P. Monk, Inverse scattering from an orthotropic medium,
{\em J. Comput. Appl. Math.} {\bf 81} (2007), 269-298.

\bibitem{SP02} S.N. Chandler-Wilde and R. Potthast, The domain derivative in rough-surface scattering and rigorous estimates for first-order perturbation theory, {\em Proc. R. Soc. Lond.\bf A458} (2002), 2967-3001. 

\bibitem{SP05} S.N. Chandler-Wilde and A.T. Peplow, A boundary integral equation formulation for the helmholtz equation in a locally perturbed half-plane, {\em ZAMM-J. Appl. Math. Mech., \bf85}(2005), 79-88.

\bibitem{SRZ99} S.N. Chandler-Wilde, C.R. Ross and B. Zhang, Scattering by infinite one-dimensional rough
surfaces, {\em Proc. R. Soc. Lond. \bf A455} (1999), 3767-3787.

\bibitem{SZ99} S.N. Chandler-Wilde and B. Zhang, Scattering of electromagnetic waves by rough interfaces
and inhomogeneous layers, {\em SIAM J. Math. Anal. \bf 30} (1999), 559-583.

\bibitem{DLLY17} M. Ding, J. Li, K. Liu and J. Yang, Imaging of locally rough surfaces by the linear sampling
method with the near-field data, {\em SIAM J. Imaging Sci. \bf10}(2017), 1579-1602.


\bibitem{KT00} R. Kress and T. Tran, Inverse scattering for a locally perturbed half-plane, {\em Inverse Problems \bf 16}(2000), 1541-1559.

\bibitem{AL08} A. Lechleiter, Factorization Methods for Photonics and Rough Surfaces, PhD Thesis,
KIT, Germany, 2008.

\bibitem{C03} C. Lines, Inverse scattering by unbounded rough surfaces, PhD Thesis,
Department of Mathematics, Brunel University, U.K., 2003.

\bibitem{LC05} C. Lines and S.N. Chandler-Wilde, A time domain point source method for inverse scattering
by rough surfaces, {\em Computing \bf75} (2005), 157-180.

\bibitem{LLLL15} J. Li, P. Li, H. Liu and X. Liu, Recovering multiscale buried anomalies in a
two-layered medium, {\em Inverse Problems} {\bf 31} (2015), 105006.

\bibitem{LYZ13} J. Li, G. Sun and R. Zhang, The numerical solution of scattering by infinite rough
interfaces based on the integral equation method, {\em Comput. Math. Appl.} {\bf71} (2016), 1491-1502.



\bibitem{LYZ18} J. Li, J. Yang and B. Zhang, A linear sampling method for inverse acoustic scattering by
a locally rough interface, arXiv: 2008.01353.


\bibitem{L10} P. Li, Coupling of finite element and boundary integral methods for electromagnetic
scattering in a two-layered medium, {\em J.  Comput. Phys. \bf 229} (2010), 481-497.

\bibitem{LZZ18} X. Liu, B. Zhang and H. Zhang, A direct imaging method for inverse scattering by unbounded
rough surfaces, {\em SIAM J. Imaging Sci. \bf11} (2018), 1629-1650.



\bibitem{DTS03} D. Natroshvili, T. Arens and S.N. Chandler-Wilde, Uniqueness, existence, and integral
equation formulations for interface scattering problems,
{\em Memoirs Differ. Equat. Math. Phys. \bf 30} (2003), 105-146.

\bibitem{QZZ19} F. Qu, B. Zhang and H. Zhang, A novel integral equation for scattering by locally rough
surfaces and application to the inverse problem: The Neumann case,
{\em SIAM J. Sci. Comput. \bf41} (2019), A3673-A3702.

\bibitem{RM15} D.G. Roy and S. Mudaliar, Domain derivatives in dielectric rough surface scattering,
{\em IEEE Trans. Antennas Propagation \bf63} (2015), 4486-4495.

\bibitem{MT06} M. Thomas, Analysis of Rough Surface Scattering Problems, PhD Thesis.
The University of Reading, UK, 2006.

\bibitem{WW87} A. Willers and P. Werner, The helmholtz equation in disturbed half-spaces, {\em Math. Method Appl. Sci.,\bf9}(1987), 312-323.

\bibitem{XZZ19} X. Xu, B. Zhang and H. Zhang, Uniqueness and direct imaging method for inverse scattering
by locally rough surfaces with phaseless near-field data, {\em SIAM J. Imaging Sci. \bf12} (2019), 119-152.

\bibitem{YLZ20} J. Yang, J. Li and B. Zhang, Simultaneous recovery of a locally rough interface
and its buried obstacles and homogeneous medium, preprint.





\bibitem{ZZ2017} B. Zhang and H. Zhang, Imaging of locally rough surfaces from intensity-only far-field
or near-field data, {\em Inverse Problems \bf33} (2017) 055001 (28pp).

\bibitem{ZZ13} H. Zhang and B. Zhang, A novel integral equation for scattering by locally rough surfaces
and application to the inverse problem, {\em SIAM J. Appl. Math. \bf73} (2013), 1811-1829.

\bibitem{Zhang20} H. Zhang, Recovering unbounded rough surfaces with a direct imaging method,
{\em Acta Mathematicae Applicatae Sinica, English Series \bf36} (2020), 119-133.

\end{thebibliography}
\end{document}